\documentclass[review,1p,number]{elsarticle}

\usepackage{amsmath,amssymb,amsthm,color}
\usepackage[shortlabels]{enumitem}
\usepackage{graphicx}
\usepackage{graphbox}
\usepackage{refcount}

\newcommand{\conv}{\mathrm{conv}}
\newcommand{\Pf}{\mathcal{P}_{\hspace{-0.2mm}\mathrm{fin}}  }
\newcommand{\Pfs}{\Pf^*}

\renewcommand{\Re}{\mathbb{R}}

\newcommand{\sB}{\mathcal{B}}

\newcommand{\dinf}{d_{\scalebox{0.5}{\text{$\infty$}}} \hspace{-0.3mm}}
\newcommand{\hx}[2]{h_{#1}^{\hspace{-1mm}^{(#2)}}}

\graphicspath{{./Figures/}{./}}

\newcommand{\diam}{\mathrm{diam}}

\theoremstyle{plain}
\newtheorem{thm}{Theorem}[section]
\newtheorem{prop}[thm]{Proposition}
\newtheorem{cor}[thm]{Corollary}

\newtheorem{lem}[thm]{Lemma}

\theoremstyle{definition}

\newtheorem{example}[thm]{Example}

\begin{document}

\begin{frontmatter}

\title{Subtree Distances, Tight Spans and Diversities}

\author[1]{David Bryant\corref{cor1}}
\ead{david.bryant@otago.ac.nz}

\author[2]{Katharina T. Huber}
\ead{K.Huber@uea.ac.uk}
\author[2]{Vincent Moulton}
\ead{v.moulton@uea.ac.uk}
\author[3]{Andreas Spillner}
\ead{andreas.spillner@hs-merseburg.de}
\affiliation[1]{organization={Department of Mathematics and Statistics, University of Otago}, addressline={PO Box 56}, postcode=9054,city={Dunedin}, country={New Zealand}}
\affiliation[2]{organization={School of Computing Sciences, University of East Anglia}, postcode={NR4 7TJ}, cite={Norwich}, country={United Kingdom}}
\affiliation[3]{organization={Merseburg University of Applied Sciences}, postcode={06217}, city={Merseburg}, country={Germany}}

\date{\today}

\begin{abstract}
We characterize when a set of distances $d(x,y)$ between elements  in a set $X$ have a \emph{subtree representation},  a real tree $T$ and a collection $\{S_x\}_{x \in X}$ of subtrees of~$T$ such that $d(x,y)$ equals the length of the shortest path in~$T$ from a point in $S_x$ to a point in $S_y$ for all $x,y \in X$. The characterization was first established for {\em finite} $X$ by Hirai (2006) using a tight span construction defined for {\em distance spaces}, metric spaces without the triangle inequality.  To extend Hirai's result beyond finite $X$ we establish fundamental results of tight span theory for general distance spaces, including the surprising observation that the tight span of a distance space is hyperconvex. We apply the results to obtain the first characterization of when a diversity -- a generalization of a metric space which assigns values to all finite subsets of~$X$, not just to pairs -- has a tight span which is tree-like. 
\end{abstract}


\newpageafter{author}
\end{frontmatter}

\section{Introduction}
\label{sec:intro}

A classical result of metric geometry is that a 
metric space $(X,d)$ can be embedded in a tree if and only if it satisfies the {\em four-point condition:} 
\begin{equation} \label{eq:4pt}
d(x,y)+d(w,z) \leq \max \{d(w,x) + d(y,z),d(x,z)+d(w,y)\}
\end{equation}
for all $w,x,y,z \in X$ \cite{Dress84,Pereira69}.


Hirai \cite{Hirai06} proved an appealing generalization of this result
for finite \emph{distance spaces} \((X,d)\),
where $X$ is a finite, non-empty set
and \(d : X \times X \rightarrow \mathbb{R}_{\geq 0}\)
is symmetric, vanishes on the diagonal but need not satisfy the
triangle inequality. 
Distance spaces are also known as semi-pseudometrics \cite[p.~300]{cech1966} and smetrics \cite{SmythTsaur02}.

A collection of subtrees $\{S_x\}_{x \in X}$ of an edge-weighted tree determines a distance $d$ on $X$ where, for all $x,y \in X$, $d(x,y)$ equals the length of the shortest path in the tree between $S_x$ and $S_y$. 
If an arbitrary distance $d$ on $X$ can be represented in this way, we say that the collection $\{S_x\}_{x \in X}$ is a {\em subtree representation} of $d$. 
Hirai  proved that a finite distance space $(X,d)$ has such a subtree representation if and only if it satisfies the {\em extended four-point condition}: 
\begin{equation}
\label{eq:extended4pt}
d(x,y) + d(z,w) \leq \max \left\{ 
\begin{array}{l}
d(x,y), \ d(w,z),\\
d(w,x) + d(y,z) , \ d(x,z) + d(w,y), \\ 
(d(x,y) + d(y,z) + d(z,x))/2, \\
(d(x,y) + d(y,w) + d(w,x))/2,\\  
(d(x,z) + d(z,w) + d(w,x))/2, \\ 
(d(y,z) + d(z,w) + d(w,y))/2 
\end{array} \right\}\end{equation}
for all $w,x,y,z \in X$. The distances between subtrees in a tree do not, in general, obey the triangle inequality, and two distinct subtrees can have distance zero (see Figure~\ref{fig:ex:intro:std} for an example).  Condition~\ref{eq:extended4pt} reduces to the classical four-point condition when $(X,d)$ is a metric space. 

\begin{figure}
\begin{center}
\begin{tabular}{cccc}
(a) & & (b) &\\
& \includegraphics[scale=0.7,align=c]{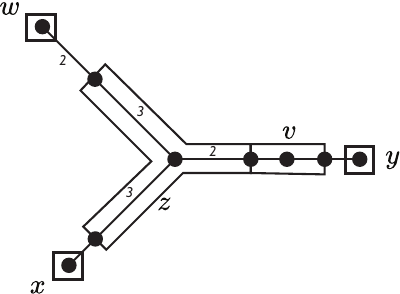} & & {\small $\begin{array}{c|ccccc} d& x& y& z& v & w \\ \hline x&0 & 9 & 1 & 6 & 9 \\y& 9 & 0 & 3 & 1 & 10 \\z& 1 & 3 & 0 & 0 & 2 \\v& 6 & 1 & 0 & 0 & 7 \\w& 9 & 10 & 2 & 7 & 0 \end{array}$}
\end{tabular}
\end{center}
\caption{\small (a) A subtree representation of a distance $d$ on $X=\{v,w,x,y,z\}$. 
The subtrees representing $w$, $x$ and $y$ each consist of a single leaf of the underlying
edge-weighted tree. The subtree representing $v$ consists of
two edges and the subtree representing $z$ consists of three edges
(only edge-weights $\neq 1$ are shown).
(b) The table of the distances between the subtrees in~(a).
We have $d(v,z)=0$ and the triangle inequality
is violated since $d(x,y) > d(x,z) + d(z,y)$.}
\label{fig:ex:intro:std}
\end{figure}

The main tool that Hirai uses to prove the characterization is the {\em tight span}. The tight span of a metric space, otherwise known as the injective or hyperconvex hull,  is a fundamental object in metric space theory.
It was  introduced by Aronszajn and Panitchpakdi~\cite{Aronszajn56}
and Isbell~\cite{Isbell64}, and studied extensively 
by  Dress et al. \cite{dress1996t,Dress84}. 
It has applications in areas such as combinatorial optimization \cite{hirai2011folder}, tropical geometry \cite{develin2004tropical}, and group theory \cite{lang2013injective}.   

Let $\Re^X$ denote the set of real-valued functions on $X$ and define
\begin{equation}
\label{eq:def:d:infty}
\dinf(f,g) := \sup\{|f(x) - g(x)|: x \in X\} 
\end{equation}
for $f,g \in \Re^X$, noting that this supremum can be infinite.
The tight span $T_\rho$ of a metric space $(X,\rho)$ consists of the point-wise minimal elements in the set 
\begin{equation}
P_\rho  := \{f \in \Re^X :  f(x) + f(y) \geq \rho(x,y) \mbox{ for all } x,y \in X\} \label{eq:def_Prho}.
\end{equation}
The pair $(T_\rho,\dinf)$ is  a metric space. It includes, for each $x \in X$, the function 
\begin{equation}
\label{eq:kuratowski:td:metric}
\hx{x}{\rho}:X \rightarrow \Re:y \mapsto \rho(x,y).
\end{equation}
The map $\kappa:X \rightarrow \Re^X$ taking $x$ to $\hx{x}{\rho}$ is an isometric embedding of $(X,\rho)$ into $(T_\rho,\dinf)$, known as the {\em Kuratowski embedding}.

Though the theory of tight spans was  developed for metric spaces, 
various approaches have been presented to extend the theory to distance spaces. 
In \cite{Hirai06,hirai2006geometric} Hirai develops tight span theory for finite distance spaces, an
approach that was built upon in for example \cite{HerrmannMoulton12} (see also \cite[Chap. 5]{dress2011basic}). 
An alternative definition for injectivity of distance spaces is studied in  \cite{SmythTsaur02}. \\

In this paper, as in \cite{Hirai06}, the 
tight span $T_d$ of a distance space $(X,d)$ is defined as the pointwise-minimal elements in the set 
\begin{equation}
P_d := \{f \in \Re^X : f(x) + f(y) \geq d(x,y) \mbox{ for all } x,y \in X\}  \label{eq:def_Pd},
\end{equation}
just as for metric spaces. 
The absence of a triangle inequality implies that the function $\hx{x}{d}$ need not be an element of $T_d$, nor  of $P_d$. Instead,  $\kappa$ is defined as the map taking $x \in X$ to the {\em subset} of $T_d$
\begin{equation}
\label{eq:kappa:x}
\kappa(x) := \{f \in T_d:f(x) = 0\}.
\end{equation}
Hirai proves that, when $X$ is finite, 
\begin{equation} \label{eq:kappaEmbed}
d(x,y) = \inf\{\dinf(f,g):f \in \kappa(x),\, g \in \kappa(y)\}
\end{equation}
for all $x,y \in X$. This result is central to the characterization of subtree distances. 

One of our main contributions is to extend Hirai's results from finite $X$ to general $X$. As we shall see (Section~\ref{sec:exists}), even basic results about the tight span of a distance space fail when $X$ is not finite. There are cases when the tight span of a distance space is empty, and others where there are $f,g \in T_d$ with $\dinf(f,g) = \infty$. We show that Hirai's extended four-point condition for subtree representations holds for general distance spaces $(X,d)$, not just finite ones (Theorem~\ref{thm:char:like:hirai}). To prove this, we  establish fundamental properties of the distance tight span as well as  pathological cases when the construction breaks down.  Like Hirai, we make use of the distance tight span to characterize subtree representations, though  follow quite a different strategy. 

Our original motivation for studying tight spans of distances came from diversity theory~\cite{BryantTupper12}. A {\em diversity} $(X,\delta)$, formally defined in Section~\ref{sec:Diversities}, can be regarded as generalization of a metric space $(X,d)$, where~$\delta$ assigns values to all finite subsets of~$X$, not just to pairs. The axioms for diversities parallel those for metric spaces, as do the theorems and applications. The mathematics of diversity tight spans \cite{BryantTupper12} extends that for metric tight spans to a surprising extent. As we shall 
see, tight spans of diversities are closely related to tight spans of distance spaces (Theorem~\ref{thm:TdeltaEmbed}, Theorem~\ref{thm:nice}). 
As an application of our main result, we shall present
a characterization of those diversities whose tight span is a tree (which we call 
{\em arboreal diversities}) using the results on subtree representations (Theorem~\ref{thm:char:arboreal}).\\

The remainder of this paper is structured as follows:\\

\noindent {\em Section 2} develops the theory of tight spans for general, possibly infinite, distance spaces. The starting point is tight span theory for metric spaces, so we begin with a concise summary of relevant results for metrics, many of which reappear later in a different form.

A key tool for working with the tight span of a distance space $(X,d)$ is the set of metric spaces which dominate it, that is, metric spaces $(X,\rho)$ for which $d(x,y) \leq \rho(x,y)$ for all $x,y \in X$. There are useful connections (e.g. Proposition~\ref{prop:PdTd}) between the tight span of a distance function and the (metric) tight spans of its dominating metrics. We use these relationships to prove Theorem~\ref{thm:kappaEmbed}, which describes how a distance space embeds into its tight span. The result  generalizes Theorem 2.4 in \cite{Hirai06} to general $X$ and plays a key part of the subtree representation proof. 

The tight span of a metric space $(X,\rho)$ is exactly the minimal hyperconvex metric space into which $(X,\rho)$ can be embedded \cite{EspinolaKhamsi01}. Hyperconvex metric spaces have many important characteristics and properties and it is surprising that so much structure emerges from such a simple and general tight span construction. It is even more surprising that such a structure emerges even when we start with a distance space rather than a metric space. It is proved in Theorem~\ref{thm:TdHyperconvex}, that the tight span of a distance space is hyperconvex (with caveats). Our proof draws heavily on a characterization by Descombes and Pav\'on~\cite{DescombesPavon17} of hyperconvex subsets of $\ell_\infty$-spaces. 

The hyperconvexity result is useful for the subtree representation theorem, however it also reveals deeper properties of the distance tight span. We show, for example, that the {\em distance}  tight span is equal to the {\em metric} tight span of the metric space formed from the union of sets $\kappa(x)$ with which we embed the distance space into the tight span. We also show  (Proposition~\ref{cor:minimal:dom:metrics:general} and Theorem~\ref{thm:TdEmbedMinimal}) that minimal metric spaces dominating a distance space, as well as their tight spans, can be embedded in the distance tight span. 

We note that many of the results in this section come with an implicit caveat. In some cases, the tight span of a distance space can be empty. In others, the tight span can be non-empty but there exist elements at infinite distance from each other. In Section~\ref{sec:exists}, we provide examples of these and some characterizations. There is scope for more direct necessary and sufficient conditions of when a distance space has a non-trivial tight span.\\

\noindent {\em Section 3} generalizes Hirai's subtree representation theorem for arbitrary distance spaces using real trees. Our approach differs quite significantly from that used by Hirai. For a start, we need to demonstrate that a distance space satisfying the extended four-point condition has a non-empty tight span that is a metric space, something which is not needed for finite distance spaces. 

A key observation we make use of is that if a distance space satisfies the extended four-point condition~\eqref{eq:extended4pt} then so does its tight span and, as the tight span is a metric space, it follows that the tight span satisfies the classical four-point condition and is therefore a real tree. With that established, the main result (Theorem~\ref{thm:char:like:hirai}) follows once we establish that the sets $\kappa(x)$ are indeed closed (and geodesic) subtrees, something we prove as a consequence of hyperconvexity. 

We strengthen Hirai's result further by showing that the representation provided by the distance tight span is minimal, in the sense that it can be embedded into any other subtree representation of the same distance space (Theorem~\ref{thm:td:minimal:subtree:representation}). \\

\noindent {\em Section~4} switches from metric spaces to {\em diversities}. A diversity is a generalization of a metric to a function defined on finite subsets, rather than just pairs. Diversities have only been introduced fairly recently \cite{BryantTupper12} and yet have already exhibited a rich theory and   `remarkable'  analogies between hyperconvexity theory for diversities and metric spaces \cite{KirkShahzad14}. In particular there is a well-developed theory of diversity tight spans and hyperconvexity, briefly reviewed at the beginning of Section~4.

This section develops tools for working with diversity tight spans which make direct use of our results for distance tight spans. Given a diversity on $X$ we define an associated distance space on the set of finite subsets of $X$ and then demonstrate that we can embed a diversity tight span into the corresponding distance tight span (Theorem~\ref{thm:TdeltaEmbed}). In general, this embedding is not bijective, but there are situations where it is and we determine those in Theorem~\ref{thm:nice}.\\

In {\em Section~5} we describe the  problem which was, for us, the catalyst for the work on distance tight spans and subtree representations. We determine necessary and sufficient conditions for when the tight span of a diversity is a tree (when considered as a metric space). The analogous questions for metric spaces was answered by \cite{Dress84}, and indeed led to a host of applications of related theory to computational biology \cite{BandeltDress92,BryantMoulton04,HusonBryant06}. The characterization we prove (Theorem~\ref{thm:char:arboreal}) is based on the distance space introduced in Section~4 and the main equivalency theorem for subtree representation (Theorem~\ref{thm:char:like:hirai}). \\

\section{Tight spans of distance spaces}
\label{sec:tspan:basics}

\subsection{Background}

A pair $(X,d)$ is a \emph{distance space} if $X \neq \emptyset$ and $d:X \times X \rightarrow \Re_{\geq 0}$ is a symmetric map that vanishes on the diagonal. The space is {\em non-null} if $d(x,y) > 0$ for some $x,y \in X$. A \emph{metric space}~$(X,\rho)$ is a distance space which, in addition, satisfies the triangle inequality 
\begin{equation}
\label{eq:metric:triangle:ineq}
\rho(x,z) \leq \rho(x,y) + \rho(y,z)
\end{equation}
for all $x,y,z \in X$. Note that, for convenience, we do not require separability, so that we can have $\rho(x,y) = 0$ even when $x \neq y$. 

For $f,g \in \Re^X$ we write $f \preceq g$ if $f(x) \leq g(x)$ for all $x \in X$. The following theorem summarizes key properties of the tight span $(T_\rho,d_\infty)$ of a metric space $(X,\rho)$. We continue to use the notation introduced in Section~\ref{sec:intro}.

\begin{thm} {\rm (\cite{Dress84})} \label{thm:metricTd}
Let $(X,\rho)$ be a metric space. For each $x \in X$ define $\kappa(x) = \hx{x}{\rho}$.
\begin{enumerate}[(i)]
\item For all $x \in X$, $\kappa(x) \in T_\rho$.
\item If $f \in T_\rho$ and $f(x) = 0$ then $f = \kappa(x)$.
\item $f \in T_\rho$ if and only if 
\[f(x) = \sup\{\rho(x,y) - f(y):y \in X\}\] 
 for all $x \in X$.
 \item For all $f,g \in T_{\rho}$ 
 \[\dinf(f,g) = \sup\{\rho(x,y) - f(x) - g(y):x,y \in X\}.\]
\item For all $x,y \in X$ and $f \in T_\rho$, 
\begin{align*}
f(x) & = \dinf(f,\kappa(x)) \\
d(x,y) &= \dinf(\kappa(x),\kappa(y)).
\end{align*}
\item There exists a map $\phi:P_\rho \rightarrow T_\rho$ such that $\phi(f) \preceq f$ and $\dinf(\phi(f),\phi(g)) \leq \dinf(f,g)$  
for all $f,g \in P_\rho$.
\end{enumerate}
\end{thm}

Dress \cite{Dress84} defines a metric space as {\em fully spread} if it is isometric to its tight span. The same concept has two important and equivalent formulations.  Let $\sB_\rho(x,r) = \{y \in X: \rho(x,y) \leq r\}$ denote the closed ball with center $x$ and radius $r$  in a metric space $(X,\rho)$. Then $(X,\rho)$ is {\em hyperconvex} if $\bigcap_{\alpha \in \Gamma} \sB_\rho(x_\alpha,r_\alpha) \neq \emptyset$ for any collection of points $\{x_\alpha\}_{\alpha \in \Gamma}$ and positive numbers $\{r_\alpha\}_{\alpha \in \Gamma}$ such that $\rho(x_\alpha,x_\beta) \leq r_\alpha + r_\beta$ for any $\alpha$ and $\beta$ in $\Gamma$ \cite{Aronszajn56,EspinolaKhamsi01}. 

A map $f$ from  a metric space  $(Y,\rho_Y)$ to $(X,\rho)$ is {\em non-expansive} if $\rho(f(y_1),f(y_2)) \leq \rho_Y(y_1,y_2)$ for all $y_1,y_2 \in Y$ and is an {\em  embedding} if $\rho(f(y_1),f(y_2)) = \rho_Y(y_1,y_2)$ for all $y_1,y_2 \in Y$. We say that $(X,\rho)$ is {\em injective} if for any embedding $\pi$ from a metric space $(Y,\rho_Y)$ to a metric space $(Z,\rho_Z)$ and every non-expansive map $f$ from $(Y,\rho_Y)$ to $(X,\rho)$ there is a non-expansive map $g$ from $(Z,\rho_Z)$ to $(X,\rho)$ such that $f(y) = g(\pi(y))$ for all $y \in Y$ \cite{EspinolaKhamsi01,Isbell64}.

\begin{thm} {\rm (\cite{EspinolaKhamsi01})} \label{thm:metricHyperconvex}
\begin{enumerate}
\item[(i)] A metric space $(X,\rho)$ is fully spread if and only if it is hyperconvex, if and only if it is injective. 
\item[(ii)] If there is an embedding from $(X,\rho)$ to some hyperconvex metric space $(Y,\rho_Y)$ then there is an embedding from $(T_\rho,\dinf)$ to $(Y,\rho_Y)$.
\end{enumerate}
\end{thm}
Property (ii) justifies the terms {\em hyperconvex hull} or {\em injective hull} for the tight span. 

Hyperconvex metric spaces have several geometric properties, two of which now we recall.
We say that a metric space $(X,\rho)$ is {\em geodesic} if for all $x,y \in X$ there is an isometric embedding~$g$ from the segment $[0,\rho(x,y)]$ in the real line to $(X,\rho)$ such that $g(0) = x$ and $g(\rho(x,y)) = y$. The image of $g$ is called a {\em geodesic path} from $x$ to $y$.

\begin{prop} {\rm (\cite[Sec.~9.2]{KS19a})} \label{hyperconvexProperties}
A hyperconvex metric space is complete and geodesic.
\end{prop}


\subsection{Distance tight spans and embeddings}

Let $(X,d)$ be a distance space.  Recall that, analogous to the metric case, 
\[P_d := \{f \in \Re^X: f(x) + f(y) \geq d(x,y) \mbox{ for all } x,y \in X \},\]
and  $T_d$ is the set of elements of $P_d$ which are minimal under $\preceq$. 
 Clearly, if a distance space happens to be a metric space then the `distance'  tight span is the same as its `metric'  tight span. 

As we now show, some of the basic properties of the metric tight span carry directly over to the distance case. 
\begin{thm} \label{thm:basicTd}
Let $(X,d)$ be a distance space.
\begin{enumerate}[(i)]
\item For all $f \in P_d$ there is $g \in T_d$ such that $f \preceq g$.
\item For $f \in \mathbb{R}^X$ we have $f \in T_d$ if and only if for all $x \in X$,
\[f(x) = \sup\{d(x,y) - f(y):y \in X\}.\]
(This is Lemma~5.1 in \cite{dress2011basic})
\item For $f,g \in T_d$, 
\[\dinf(f,g) = \sup\{d(x,y) - f(x) - g(y) : x,y \in X\}.\]
\end{enumerate}
\end{thm}
\begin{proof}
\noindent (i) Let $\mathcal{H} = \{h \in P_d: h \preceq f\}$. For each $x$ the set $\{h(x):h \in \mathcal{H}\}$ is closed and any descending chain $h_1 \succeq h_2 \succeq \cdots$ in $\mathcal{H}$ has a lower bound $\ell \in \mathcal{H}$ given by the  pointwise infinum 
\[\ell(x) = \inf\{h_i(x):i=1,2,\ldots\}.\]
Let $g$ be a minimal element of $\mathcal{H}$, which exists by Zorn's lemma. \\

\noindent  (ii)   For all $x,y$ we have $f(x) \geq d(x,y) - f(y)$ so $f(x) \geq  \sup\{d(x,y) - f(y) : y \in X\}$. As $f$ is minimal, for all $\epsilon>0$ there is $y$ such that $d(x,y) \leq f(x) + f(y) < d(x,y) + \epsilon$, giving $f(x) < d(x,y) - f(y) + \epsilon$. Taking $\epsilon \rightarrow 0$ gives the result.\\

\noindent (iii)  For all $x,y \in X$,
\begin{align*}
f(x) - g(x) & =  f(x) + f(y) - f(y) - g(x) \\
& \geq d(x,y) - f(y) - g(x),
\end{align*}
so $\dinf(f,g) \geq \sup\{ d(x,y) - f(x) - g(y) : x,y \in X\}$.  This holds even when the supremum is infinite.

For the reverse inequality, we consider two cases. First suppose 
\[\dinf(f,g) < \infty.\]
Let $\epsilon >0$. From the definition of $\dinf$ there is $x \in X$ such that
\begin{equation}
 \dinf(f,g) -\epsilon <   |f(x) - g(x)| \leq \dinf(f,g). \label{eq:distInequality1}
 \end{equation}
Without loss of generality we assume $f(x) \geq g(x)$. As $f$ is minimal, there is $y$ such that 
\[0 \leq f(x) + f(y) - d(x,y) <  \epsilon\]
and so
\begin{equation}
-\epsilon < d(x,y) - f(x) - f(y) \leq 0. \label{eq:distInequality2}
\end{equation}

Adding \eqref{eq:distInequality1} and \eqref{eq:distInequality2} with $f(x) \geq g(x)$  gives
\[ \dinf(f,g) -2\epsilon < d(x,y) - f(y) - g(x)  \leq  \dinf(f,g).\]
Taking $\epsilon \rightarrow 0$ gives the result.

For the second case, suppose $\dinf(f,g) = \infty$. Fix $\epsilon>0$. Given any $K \in \Re$ there is $x$ such that (switching $f$ and $g$ if necessary)  $f(x) - g(x) > K+\epsilon$.  Assume $f(x) > g(x)$. There is $y$ such that $d(x,y) \leq f(x) + f(y) < d(x,y) + \epsilon$, or $d(x,y) - f(x) - f(y) > -\epsilon$. It follows that $d(x,y) - g(x) - f(y) > K$, so that 
\[\sup\{ d(x,y) - f(x) - g(y) : x,y \in X\} = \infty.\]
\end{proof}

\begin{example}
We give two  examples of $T_d$ for small sets $X$. \\
(a) A minimal example of a distance space $(X,d)$ which is {\em not} a metric space is a map $d$ on three points $\{x,y,z\}$ with $d(x,y) > d(x,z) + d(y,z)$. Writing each $f:X \rightarrow \Re$ as a vector $[f(x),f(y),f(z)]$ we have that 
\begin{align*} 
T_d &=  \{[t,d(x,y)-t,d(x,z)-t]: 0 \leq t  \leq d(x,y) \} \\ &\quad \quad  \cup \{[t,d(x,y) - t,0]: d(x,z) \leq t \leq d(x,y) - d(y,z)\} \\ &\quad \quad  \cup \{[d(x,y)-t ,t,d(y,z)-t]: 0 \leq t \leq  d(y,z) \} .\end{align*} 
See Figure~\ref{fig:simpleTd}(a) for a representation of $T_d$ in the case that $d(x,z) = d(y,z) = 1$ and $d(x,y) = 3$.\\
(b)
\label{ex:td:octagon}
Let $X = \{w,x,y,z\}$ and let $d$ be the distance on $X$ given by $d(w,y) = d(x,z) = 3$ 
and $d(w,x)=d(w,z)=d(x,y)=d(y,z)=1$. The set $T_d$ is 2-dimensional and forms an octagon in~$\mathbb{R}^X$. 
See Figure~\ref{fig:simpleTd}(b) for a representation of $T_d$, noting that we use vectors $[f(w),f(x),f(y),f(z)]$ to represent functions in $T_d$.
\qed \end{example}
\begin{figure}
\centering
\includegraphics[scale=0.6]{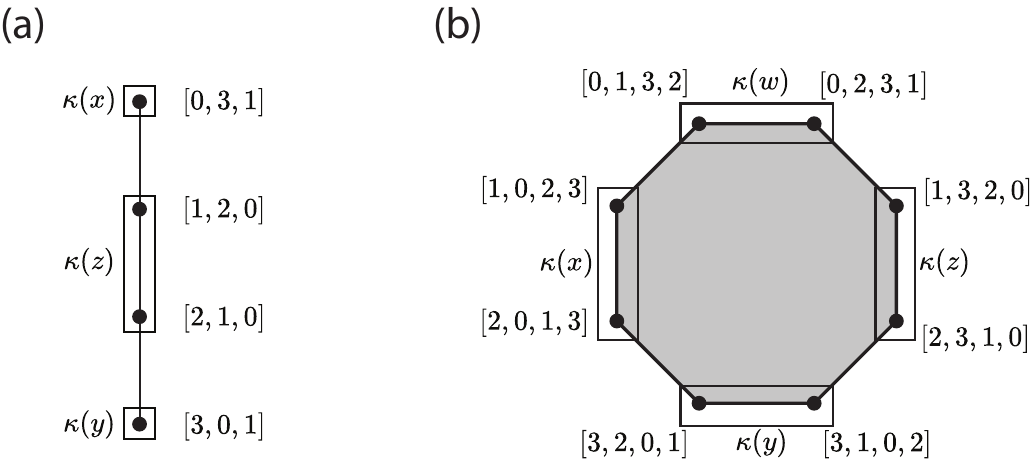}
\caption{\small \label{fig:simpleTd} Two examples of the distance tight span. (a) The tight span for the distance on $X = \{x,y,z\}$ with $d(x,z) = d(y,z) = 1$ and $d(x,y) = 3$. Functions $f$ are depicted by vectors $[f(x),f(y),f(z)]$. (b) The tight span for the distance $d$ on $X = \{w,x,y,z\}$ with $d(x,z) = d(w,y) = 3$ and all other distances equal to $1$. Functions in $T_d$ are represented by  vectors $[f(w),f(x),f(y),f(z)]$.}
\end{figure}

Given two distance functions (or metrics) $d$ and $p$ on the same set $X$ we write $d \preceq p$ if $d(x,y) \leq p(x,y)$ for all $x,y \in X$. For a distance space $(X,d)$ we define the set 
\[M(d) := \{\rho \succeq d: (X,\rho) \mbox{ is a metric space} \} \]
of metrics dominating the distance function $d$, noting (see Example~\ref{ex:noMd}) that this set can be empty.

\begin{prop} \label{prop:PdTd}
Let $(X,d)$ be a distance space. If $M(d)$ is empty then so are $P_d$ and $T_d$. Otherwise
\begin{align}
P_d &= \bigcup_{\rho \in M(d)} P_\rho \nonumber  \\
T_d &  \subseteq \bigcup_{\rho \in M(d)} T_\rho. \label{eq:TdTrho}
\end{align}
\end{prop}
\begin{proof}
If $P_d$ is non-empty then so is $T_d$, by Theorem~\ref{thm:basicTd}(i), and if $f \in T_d$ then the metric $\rho$ defined by $\rho(x,y) = f(x) + f(y)$ for all distinct  $x,y$ is an element of $M(d)$. Suppose then that $M(d)$ is non-empty. 

If $f \in P_\rho$ for some $\rho \in M(d)$ then $d(x,y) \leq \rho(x,y) \leq f(x) + f(y)$ for all $x,y$, so $f \in P_d$. Conversely, if $f \in P_d$ then the metric space $(X,\rho)$ given by $\rho(x,y) = f(x) + f(y)$ for all $x \neq y$  satisfies  $\rho \in M(d)$ and hence $f \in P_\rho$.

For the second part we have $T_d \subseteq P_d$ so $f \in T_d$ implies $f \in P_\rho$ for some $\rho \in M(d)$. Minimality of $f$ then implies $f \in T_\rho$.

\end{proof}

The inclusion in~\eqref{eq:TdTrho} can be strict, even when the union is restricted to minimal metrics $\rho \in M(d)$.

\begin{example}
\label{ex:strict:inclusion:minimal}
Let $X = \{w,x,y,z\}$ and let $d$ be the distance on $X$ with $d(x,y) = 4$ and all other pairs at distance $0$. 
Let $\rho$ be the metric given by
\[\begin{array}{c|cccc} \rho & w & x & y & z \\ \hline w & 0 & 1 & 3 & 2 \\ x & 1 & 0 & 4 & 3 \\ y & 3 & 4 & 0 & 1 \\ z & 2 & 3 & 1 & 0 \end{array}.\]
Then $\rho$ is minimal in $M(d)$. The function 
$\hx{w}{\rho}$ given by $\hx{w}{\rho}(u) = \rho(w,u)$ for all $u \in X$ is in $T_\rho$. 

Define $h:X \rightarrow \Re$ by  $h(w) = h(z) = 0$, $h(x) = 1$ and $h(y) = 3$. Then $h \in P_d$, $h \preceq \hx{w}{\rho}$ and yet $h \neq \hx{w}{\rho}$. Hence $\hx{w}{\rho}$ is not minimal in $P_d$ and
\[ \hx{w}{\rho}  \in \left(\bigcup_{\rho \in M(d)} T_\rho \right)  \setminus T_d.\]
\qed
\end{example}

We make use of the following technical lemma about $M(d)$.

\begin{lem} \label{lem:minRho} 
If $M(d)$ is non-empty and $x,y \in X$ then there is $\rho \in M(d)$ such that $\rho(x,y) = d(x,y)$. 
\end{lem}
\begin{proof}
Suppose that $p \in M(d)$ and $p(x,y) > d(x,y)$. Define $\rho$ by
        \begin{align*}
            \rho(x,y) & = d(x,y) \\
            \rho(x,u) & = p(x,u) + p(x,y) - d(x,y) \geq p(x,u)\\
            \rho(y,u) & = p(x,u) + p(x,y) \geq p(y,u) \\
            \rho(u,v) & = p(u,v)
        \end{align*}
        for all $u,v \in X \setminus \{x,y\}$, see Figure~\ref{fig:minimalRho}. Then $\rho \in M(d)$ and $\rho(x,y) = d(x,y)$. 
  \end{proof}
  
\begin{figure}
\centering
\includegraphics[scale=0.6]{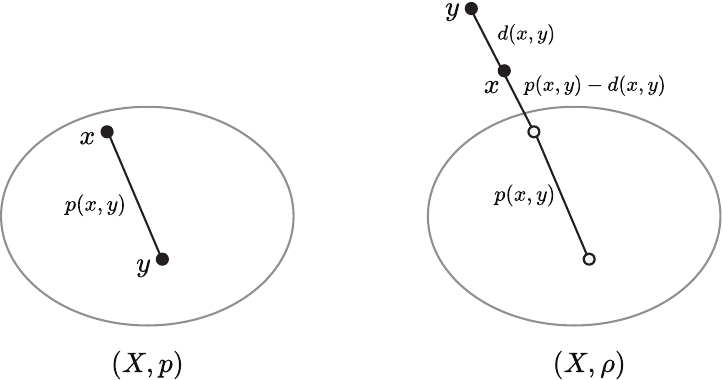}
\caption{\small \label{fig:minimalRho} Editing a metric space $(X,p)$ to give the metric space $(X,\rho)$: we remove $y$ and attach it to $x$ with a line segment of length $p(x,y)$ and then move $x$ along that segment until it is distance $d(x,y)$ from $y$. The resulting metric space satisfies $\rho(x,y) = d(x,y)$ and $\rho(u,v) \geq p(u,v)$ for all other pairs $u,v \in X$. \label{MD_has_min}}
\end{figure}

For a distance space $(X,d)$ and $x \in X$, Hirai \cite{Hirai06} defines the set 
\[\kappa(x) := \{f \in T_d : f(x) = 0\}.\]

By Proposition~\ref{prop:PdTd}, for each $f \in \kappa(x)$ there is a metric $\rho \in M(d)$ such that $f \in T_{\rho}$, in which case $f = \hx{x}{\rho}$ by Theorem~\ref{thm:metricTd}(ii). As $g \in T_d$ and $g \preceq f$ implies $g(x) = 0$ we have that $\kappa(x)$ consists of all minimal elements in  
$\{ \hx{x}{\rho} : \rho \in M(d) \}$. As a consequence, if $(X,d)$ is a metric space, $\kappa(x)$ contains the single function $\hx{x}{d}$. 

The following theorem generalizes Theorem~\ref{thm:metricTd}(v) from metric spaces to distance spaces  and is central to our characterization of subtree distance spaces later in the article.

\begin{thm} \label{thm:kappaEmbed}
Suppose that $M(d)$ is non-empty. For all $x,y \in X$ and $f \in T_d$ we have
\begin{enumerate}[(i)]
\item
\[f(x)  = \inf \{ \dinf(f,g): g \in \kappa(x) \}  \] and
\item
\[d(x,y)  = \inf\{ \dinf(f,g): f \in \kappa(x),\, g \in \kappa(y) \}. \]
\end{enumerate}
\end{thm}
\begin{proof}
(i) For all $g \in \kappa(x)$ we have $\dinf(f,g) \geq f(x) - g(x) = f(x)$, so
\[ f(x) \leq  \inf \{ \dinf(f,g): g \in \kappa(x) \} .\]

By Proposition~\ref{prop:PdTd} there is $\rho \in M(d)$ such that $f \in T_\rho$. Then $\hx{x}{\rho}(x) = 0$ and, by Theorem~\ref{thm:metricTd}(v) we have  $f(x) = \dinf(f,\hx{x}{\rho})$. As $\hx{x}{\rho} \in P_d$ there is $g \in T_d$ such that $g \preceq \hx{x}{\rho} $. Then $g \in \kappa(x)$ and $g(y) \leq \rho(x,y)$ for all $y$.

Suppose  that $\dinf(f,g) > f(x)$, so there is $y \in X$ such that $|g(y) - f(y)| > f(x)$. Since
\[g(y) - f(y) \leq \hx{x}{\rho}(y) - f(y) \leq \dinf(f,\hx{x}{\rho}) = f(x)\]
it follows that 
\[f(y) - g(y) > f(x).\]
Let $\epsilon  = f(y) - g(y) - f(x)>0$.  Since $f \in T_d$ there is $z$ such that $f(y) + f(z) < d(y,z) + \epsilon$ and
\begin{align*}
    g(y)+g(z) &\leq  g(y) + \rho(x,z) \\
    & \leq g(y) + f(x) + f(z) \\
    & = g(y) + f(x) + f(z) + (f(y) - g(y) - f(x)) - \epsilon \\
    & = f(y) + f(z) - \epsilon \\
    & < d(y,z),
\end{align*}
contradicting $g \in T_d$.    \\

(ii) Let $x,y \in X$. First, note that if
$f \in \kappa(x)$ and $g \in \kappa(y)$, then $f(x)=0$, $g(y)=0$ and 
\[\dinf(f,g) \geq g(x) - f(x) \geq d(x,y) - g(y) = d(x,y).\]

It remains to show that there exist $f \in \kappa(x)$ and 
$g \in \kappa(y)$ with $\dinf(f,g) \le d(x,y)$.
    
By Lemma~\ref{lem:minRho} there exists some $\rho \in M(d)$ such that $\rho(x,y) = d(x,y)$. Let $h = \hx{y}{\rho}$, so  $h \in T_\rho \subseteq P_d$. By Theorem~\ref{thm:basicTd}(i), there is some $f \in T_d$ with $f \preceq h$. Hence, $f(y) = h(y) = 0$ and $f(x) \leq h(x) = \rho(x,y) = d(x,y)$. Thus $f \in \kappa(y)$, and it follows from (i) that there is some $g \in \kappa(x)$ with  $\dinf(f,g) = f(x) \le d(x,y)$, as required.

 \end{proof}

\subsection{Distance tight spans are hyperconvex}

When $(X,d)$ is a metric space, the tight span $(T_d,\dinf)$ is hyperconvex, in fact it functions as a minimal {\em hyperconvex hull} for $(X,d)$. When $(X,d)$ is a distance space, its tight span $(T_d,\dinf)$ is not necessarily a metric space: sometimes $T_d$ is empty, and other times we can have $\dinf(f,g) = \infty$ for some $f,g \in T_d$. Our next result shows that if $(X,d)$ is a distance space and the tight span $(T_d,\dinf)$ of $(X,d)$ is a metric space then $(T_d,\dinf)$ is hyperconvex. This is surprising: a distance space is such a general concept, while hyperconvex metric spaces are full of structure.  Our proof is based on a characterization of hyperconvex subspaces of $\ell_\infty(X) = \{f \in \Re^X:\|f\|_\infty < \infty\}$,  due to \cite{DescombesPavon17}.

\begin{thm}
Let $(X,d)$ be a distance space. If $(T_d,\dinf)$ is a metric space then it is hyperconvex. \label{thm:TdHyperconvex}
\end{thm}
\begin{proof}
Fix $h \in T_d$ and define $T_d \! -\! h = \{f-h:f \in T_d\}$. Then $(T_d\!-\!h,\dinf)$ is isometric to $(T_d,\dinf)$ and $\|f - h\|_\infty = \dinf(f,h) < \infty$ for all $f \in T_d$, so $(T_d \!-\!h) \subseteq \ell_\infty(X)$. By Theorem~\ref{thm:basicTd}(ii) we have $f \in T_d$ if and only if for all $i \in X$
\[f(i) = \sup\{d(i,j) - f(j):j \in X\},\]
if and only if for all $i \in X$
\[f(i) = \max\left[ \sup \left\{d(i,j) - f(j):j \in X \setminus \{i\} \right\},0 \right].\]
Hence $g \in T_d\!-\!h$ if and only if for all $i \in X$, 
\[g(i) =  \max\left[ \sup \left\{d(i,j) - g(j) - h(j) :j \in X \setminus \{i\} \right\},0 \right]  - h(i) .\]

For each $i \in X$, let $\widehat {\pi}_{i}$ denote the restriction map from $\Re^{X}$ to $\Re^{X \setminus \{i\}}$ given by 
\[ \widehat{\pi}_i (f)(j) = f(j) \mbox{ for all $j \in X \setminus \{i\}$ }.\]
We define $r_i: \ell_\infty(X \setminus \{i\}) \rightarrow \Re$ by 
\[r_i(\pi) = \max\left[ \sup \left\{d(i,j) - \pi(j) - h(j) :j \in X \setminus \{i\} \right\},0 \right]  \]
and define $\underline{r}_i = \overline{r}_i = r_i$ for all $i$.
As $h \in P_d$ we have for each $g \in\ell_\infty(X)$ that
\[ d(i,j) - g(j) - h(j) \leq (h(i) + h(j)) - g(j)  - h(j) \leq h(i) - g(j) \leq h(i) + \|g\|_\infty,\]
so $r_i (\widehat{\pi}_{i} (g))  < \infty$.

Let $I = I_- = I_+ = X$. In the notation of Theorem 1.2 of \cite{DescombesPavon17} we then have $Q_- = Q_+ = Q$ where
\begin{align*}
Q & = \{g \in \ell_\infty(X) : r_i(\widehat{\pi}_{i} (g)) = g(i) \mbox{ for all } i \in X\} \\
& =  \{g \in \ell_\infty(X) : g(i) =  \max\left[ \sup \left\{d(i,j) - g(j) - h(j) :j \in X \setminus \{i\} \right\},0 \right]  - h(i) \mbox{ for all } i \in X \} \\
& = T_d - h.
\end{align*}
To apply Theorem 1.2 of  \cite{DescombesPavon17}  and show that $(T_d - h,\dinf)$ and hence $(T_d,\dinf)$ are hyperconvex it remains to show that the maps $r_i$ are non-expansive. 

Let $g,g' \in \ell_\infty(X \setminus \{i\})$ and define, for all $j \in X \setminus \{i\}$, 
\begin{align*}
    \alpha_j &= \max\{0, d(i,j) - g(j) - h(j) \} \\
    \beta_j &= \max\{0, d(i,j) - g'(j) - h(j) \}.
\end{align*}
Then, since 
\[ (d(i,j) - g(j) - h(j)) - (d(i,j) - g'(j) - h(j))  = g'(j) - g(j),\]
we have $|\alpha_j - \beta_j| \leq \|g-g'\|_\infty$ and, therefore, 
\begin{align*}
|r_i(g) - r_i(g')| & = |\sup\{\alpha_j:j \in X \setminus \{i\} \} - \sup\{\beta_j:j \in X \setminus \{i\} \}| \\
& \leq \|g-g'\|_\infty,
\end{align*}
as required.
\end{proof}

We next establish several  properties of the subsets~$\kappa(z)$ of~$T_d$ which we will make use of later.

\begin{prop} \label{prop:kappaProp}
Let $(X,d)$ be a distance space such that $(T_d,\dinf)$ is a metric space. For each $z \in X$ the subspace $(\kappa(z),\dinf)$ is non-empty, hyperconvex, complete and geodesic.
\end{prop}
\begin{proof}
If $T_d \neq \emptyset$ then $M(d) \neq \emptyset$ by Proposition~\ref{prop:PdTd}. For any $\rho \in M(d)$ there 
exists $h \in T_d$ such that $h \preceq \hx{z}{\rho}$, so that $h \in \kappa(z) \neq \emptyset$.

To show that $(\kappa(z),\dinf)$ is hyperconvex we slightly modify the non-expansive maps
used in the proof of Theorem~\ref{thm:TdHyperconvex} by choosing $h \in \kappa(z)$ and
letting, for all $i \in X$ and all $\pi \in \ell_\infty(X \setminus \{i\})$,
\[r_i(\pi) := \begin{cases}  \max \left [ \sup \Big\{d(i,j) - \pi(j) - h(j): j \in X \setminus \{i\} \Big\}, 0 \right ] & \mbox{ if $i \neq z$;} \\
h(z)& \mbox{ if $i = z$.}  \end{cases}\]
That $(\kappa(z),\dinf)$ is geodesic and complete now follows from Proposition~\ref{hyperconvexProperties}.
\end{proof}

 Define $K := \bigcup_{x \in X} \kappa(x)$ so that, by Proposition~\ref{prop:PdTd}, $K$ equals the set of minimal elements of 
 \[ \{ \hx{x}{\rho} : x \in X,\, \rho \in M(d) \}\]
 and $(K,\dinf)$ is a subspace of the metric space $(T_d,\dinf)$. We have shown that if $(T_d,\dinf)$ is a metric space then it is hyperconvex. We now show that, in this case, $(T_d,\dinf)$ is isometric with the {\em metric} tight span of the metric space
$(K,\dinf)$. Hence, the tight span of a distance space  is either not a metric space, or it equals a metric tight span. 

For the following we let $T_K$ denote the minimal elements in  
\[P_K := \{F \in \Re^K: F(k_1) + F(k_2) \geq \dinf(k_1,k_2) \mbox{ for all $k_1,k_2 \in K$} \},\]
so that $(T_K,\dinf)$ is the metric tight span of $(K,\dinf)$, noting the slight abuse of terminology with $\dinf$.

\begin{thm} \label{thm:TdTK}
Suppose $(X,d)$ is a distance space. If the tight span $(T_d,\dinf)$
of $(X,d)$ is a metric space then it is isometric to the tight span
 $(T_K,\dinf)$ of the metric space $(K,\dinf)$.
  \end{thm}

 \begin{proof}
 Define the map $\psi:T_d \rightarrow \Re^K$ by letting $\psi(f)$ be the function $F$ such that $F(k) = \dinf(f,k)$ for all $k \in K$. 
We note that for all $g \in K$ we have
 \[\sup\{\dinf(g,h) - F(h):h \in K\} = \sup\{\dinf(g,h) - \dinf(f,h):h \in K\} = \dinf(f,g) = F(g)\]
 so, by Theorem~\ref{thm:metricTd}(iii), $F \in T_K$.
 
 Suppose $f,g \in T_d$, $F = \psi(f)$ and $G = \psi(g)$. Then 
 \begin{align*}
 \dinf(F,G) &= \sup\{ |\dinf(f,k) - \dinf(g,k)| : k \in K  \} \\
 & \leq \dinf(f,g).
 \end{align*}
 For all $\epsilon>0$ there is (without loss of generality) $x \in X$ such that 
 \[ \dinf(f,g) -\epsilon/2 < f(x) - g(x) \leq  \dinf(f,g).\]
 By Theorem~\ref{thm:kappaEmbed} there is $h \in \kappa(x)$ such that $\dinf(g,h) < g(x) + \epsilon/2$ so that 
 \[\dinf(f,g) \leq f(x) - g(x)  +\epsilon/2 \leq \dinf(f,h) - \dinf(g,h)  + \epsilon \leq \dinf(F,G) + \epsilon.\]
 Taking $\epsilon \rightarrow 0$ we have 
 \[\dinf(F,G) = \dinf(f,g)\]
 so that $\psi$ is an isometry and $(T_d,\dinf)$ is isometric to a  subspace of $(T_K,\dinf)$.

As $(K,\dinf)$ is a subspace of $(T_d,\dinf)$ and $(T_d,\dinf)$ is hyperconvex, Theorem~\ref{thm:metricHyperconvex}(ii) implies that there is an isometric embedding from the tight span $(T_K,\dinf)$ of $(K,\dinf)$ in $(T_d,\dinf)$. As $(T_d,\dinf)$ is isometric to $(\psi(T_d),\dinf)$ and $\psi(T_d) \subseteq T_K$ we conclude that $\psi(T_d) = T_K$ and $(T_d,\dinf)$ is isometric to $(T_K,\dinf)$.
\end{proof}

As an immediate corollary we have

\begin{cor} \label{cor:TKisometry}
Let $(X,d)$ be a distance space such that $(T_d,\dinf)$ is a metric space. Then $\psi:T_d \rightarrow T_K$ given by 
\[\psi(f)(k) = \dinf(f,k)\]
is an isometry, with inverse
\[ \psi^{-1} (F)(x) = \inf_{k \in \kappa(x)} F(k). \]
\end{cor}

The next result is the distance space generalization of Theorem~\ref{thm:metricTd}(vi). It can almost be proved by taking the proof of 
(1.9) in \cite{Dress84} and replacing `metric' by `distance', as observed by \cite{Hirai09}, Lemma~2.2. The following paraphrases and adds details to the version of the proof in \cite{Dress84} which avoids Zorn's Lemma.

\begin{lem}\label{lem:contraction:pd:td}
Let $(X,d)$ be a distance space such that $(T_d,\dinf)$ is a metric space. There is $\phi:P_d \rightarrow T_d$ such that for all $f,g \in P_d$ we have $\phi(f) \preceq f$ and $\dinf(\phi(f),\phi(g)) \leq \dinf(f,g)$. 
\end{lem}
\begin{proof}  
For all $f \in P_d$ define the map $f_{\#} : X \rightarrow \mathbb{R}$ by putting
\begin{equation}
\label{eq:def:fsharp}
f_{\#}(x) = \sup \{d(x,y) - f(y):y \in Y\}
\end{equation}
for all $x \in X$. As $f(x) + f(y) \geq d(x,y)$ for all $f \in P_d$ and $x,y \in X$ we have $f_{\#}(x) \leq f(x)$ for all $x \in X$. 

Consider $f \in P_d$. Put $f^{(0)} = f$ and, for all $k > 0$,
$f^{(k)} = \frac{1}{2}(f^{(k-1)} + f^{(k-1)}_{\#})$. Then, assuming that $f^{(k-1)} \in P_d$, we have,  for all $x,y \in X$,
\begin{align*}
&f^{(k)}(x) + f^{(k)}(y)\\
&= \frac{1}{2}(f^{(k-1)}(x) + f^{(k-1)}_{\#}(x)) + \frac{1}{2}(f^{(k-1)}(y) + f^{(k-1)}_{\#}(y))\\
&= \frac{1}{2}(f^{(k-1)}(x) + f^{(k-1)}_{\#}(y)) + \frac{1}{2}(f^{(k-1)}(y) + f^{(k-1)}_{\#}(x))\\
&\geq \frac{1}{2}(f^{(k-1)}(x) + d(x,y) - f^{(k-1)}(x)) + \frac{1}{2}(f^{(k-1)}(y) + d(x,y) - f^{(k-1)}(y))\\
&= d(x,y),
\end{align*}
where the inequality in the fourth line holds in view of~\eqref{eq:def:fsharp}. Hence $f^{(k)} \in P_d$ and
\begin{equation}
\label{eq:converging:seq}
f_{\#}^{(k)} \preceq f_{\#}^{(k+1)} \preceq f^{(k+1)} \preceq f^{(k)}
\end{equation}
for all $k \in \mathbb{N}$. For all $x \in X$,
\begin{align*}
|f^{(k+1)}(x) - f_{\#}^{(k+1)}(x)|
&\leq |f^{(k+1)}(x) - f_{\#}^{(k)}(x)|\\
&= |\frac{1}{2}(f^{(k)}(x) + f^{(k)}_{\#}(x)) - f_{\#}^{(k)}(x)|\\
&= \frac{1}{2}|f^{(k)}(x) - f_{\#}^{(k)}(x)|,
\end{align*}
implying that the sequence $\{f^{(k)}\}_{k \in \mathbb{N}}$ converges pointwise to some $f^* \in P_d$. Moreover, we have $f^*_{\#} = f^*$ and, therefore, by Theorem~\ref{thm:basicTd}, $f^* \in T_d$.

Define the map $\phi : P_d \rightarrow T_d$ as $\phi(f) = f^*$. Consider $f,g \in P_d$. If $\dinf(f,g) = \infty$ we clearly have
\[\dinf(\phi(f),\phi(g)) < \infty = \dinf(f,g).\]
So, assume $\dinf(f,g) < \infty$. Then, for all $x \in X$, we have
\begin{align*}
f_{\#}(x)
&= \sup \{d(x,y) - f(y) : y \in X\}\\
&= \sup \{d(x,y) - g(y) + g(y) - f(y) : y \in X\}\\
&\leq \sup \{d(x,y) - g(y) : y \in X\} + \sup \{g(y) - f(y) : y \in X\}\\
&= g_{\#}(x) + \sup \{g(y) - f(y) : y \in X\},
\end{align*}
implying $\sup \{|f_{\#}(x) - g_{\#}(x)| : x \in X\} \leq \sup \{|g(x) - f(x)| : x \in X\}$. Hence,
\begin{align*}
&\dinf(f^{(k+1)},g^{(k+1)})\\
&= \sup \{\frac{1}{2}(f^{(k)}(x) + f_{\#}^{(k)}(x)) - \frac{1}{2}(g^{(k)}(x) + g_{\#}^{(k)}(x)) : x \in X\}\\
&\leq \frac{1}{2} \dinf(f^{(k)},g^{(k)}) + \frac{1}{2} \dinf(f_{\#}^{(k)},g_{\#}^{(k)}) \leq \dinf(f^{(k)},g^{(k)}).
\end{align*}
Thus, $\dinf(\phi(f),\phi(g)) = \dinf(f^*,g^*) \leq \dinf(f,g)$,
as required.
  \end{proof}

The following proposition can be viewed as an extension of Proposition~2.3 in \cite{Hirai09} from finite to infinite sets.

\begin{prop}
\label{cor:minimal:dom:metrics:general}
Let $(X,d)$ be a distance space such that $(T_d,d_{\infty})$ is a metric space and let $\rho \in M(d)$ be minimal with respect to $\preceq$. Then there exists a map $\psi : X \rightarrow T_d$ such that $\psi(x) \in \kappa(x)$ and $d_{\infty}(\psi(x),\psi(y)) = \rho(x,y)$ for all $x,y \in X$.
\end{prop}

\begin{proof}
Let $\rho \in M(d)$ be minimal with respect to $\preceq$. For each $x \in X$ consider the Kuratowski map $h^{(\rho)}_x \in T_{\rho} \subseteq P_d$. Using the map $\phi$ from Lemma~\ref{lem:contraction:pd:td}, we obtain a map $\psi: X \rightarrow T_d$
by putting $\psi(x) = \phi(h^{(\rho)}_x)$. Then we have $\psi(x) \in \kappa(x)$ in view of $h^{(\rho)}_x(x) = 0$ and $\phi(h^{(\rho)}_x) \preceq h^{(\rho)}_x$. Thus, we have
\[d(x,y) \leq d_{\infty}(\psi(x),\psi(y)) \leq d_{\infty}(h^{(\rho)}_x,h^{(\rho)}_y) = \rho(x,y),\]
where the first inequality holds in view of Theorem~\ref{thm:kappaEmbed}.
Hence, since $\rho$ is minimal in $M(d)$, we must have $d_{\infty}(\psi(x),\psi(y)) = \rho(x,y)$ for all $x,y \in X$.
\end{proof}

\begin{thm} \label{thm:TdEmbedMinimal}
Let $\rho$ be a minimal element of $M(d)$. Then there is an isometric embedding from $(T_\rho,\dinf)$ into $(T_d,\dinf)$.
\end{thm}
\begin{proof}
By Proposition~\ref{cor:minimal:dom:metrics:general} there is an isometric embedding $\psi$ from $(X,\rho)$ into $(T_d,\dinf)$ and, since $(T_d,\dinf)$ is hyperconvex, there is an isometric embedding $\phi$ from $(T_\rho,\dinf)$ into $(T_d,\dinf)$ such that $\phi(\hx{x}{\rho}) = \psi(\hx{x}{\rho})$ for all $x \in X$. 
\end{proof}

Example~\ref{ex:strict:inclusion:minimal} shows that the isometric embedding given by Theorem~\ref{thm:TdEmbedMinimal} cannot be strengthened to a statement of inclusion such as $T_\rho \subseteq T_d$.

\subsection{When does a distance space have a tight span?}\label{sec:exists}

By Proposition~\ref{prop:PdTd},  the distance tight span $T_d$ is non-empty if and only if the set $M(d)$ of metrics dominating $d$ is non-empty. 
Thus, we see immediately that $M(d)$ is non-empty whenever $X$ is finite, and we will see below that $M(d)$ is non-empty if $X$ is countable.  We show here that $M(d)$, and hence the tight span $T_d$, can be empty when $X$ is uncountable. 

\begin{example} \label{ex:noMd}
Let $X$ be the real interval  $[0,1]$ and define the distance function $d$ on $X$ by 
\[d(x,y) = \begin{cases} \frac{1}{|x-y|} & \mbox{ if $x \neq y$, } \\ 0 & \mbox{ otherwise.} \end{cases}\]
We claim that $M(d)$ is empty.  Suppose not, and that $\rho \in M(d)$. As $X$ is uncountable but can be covered by a countable number of closed balls
\[X \subseteq \bigcup_{n=1}^\infty \sB_\rho(0,n) \]
there is $m$ such that $\sB_\rho(0,m)$ is infinite. The set $\sB_\rho(0,m)$ contains a  limit point $x$ with respect to the standard  metric on $[0,1]$. Let $x_1,x_2,\ldots,x_n,\ldots$ be a sequence in $\sB_\rho(0,m) \setminus \{x\}$ which converges to $x$ with the standard metric. As 
\[|x_n - x| \rightarrow 0\]
we have
\[d(x_n,x) = \frac{1}{|x_n - x|} \rightarrow \infty,\]
contradicting the fact that for all $n$, 
\[d(x,x_n) \leq \rho(x,x_n) \leq \rho(0,x) + \rho(0,x_n) \leq 2m.\]
\qed\end{example}

The argument used in Example~\ref{ex:noMd} suggests a general characterization for when this holds. We say that a subset $B$ of a distance space $(X,d)$ is {\em bounded} if $\sup\{d(x,y):x,y \in B\}<\infty.$

\begin{prop} \label{prop:noTd}
A distance space $(X,d)$ is dominated by a metric if and only if there is a countable chain $X_1 \subseteq X_2 \subseteq \cdots$ of bounded subsets which cover $X$. 
\end{prop}
\begin{proof}
If there is $\rho \in M(d)$ then for any $x \in X$ the sets $\sB_\rho(x,n)$, $n=1,2,3,\ldots$ provide such a chain. 

Conversely, let $X_1 \subseteq X_2 \subseteq \cdots$ be a chain of subsets which cover $X$ so that the `diameter'
\[\diam(X_i) := \sup\{d(x,y):x,y \in X_i \}\]
is finite for all $i$. For each $x \in X$ there is a minimal index $i$ such that $x \in X_i$; we denote this index by $n(x)$. 

We define the symmetric map $\rho: X \times X \rightarrow \Re$ by $\rho(x,x) = 0$ for all $x$ and 
\[\rho(x,y) = \diam(X_{n(x)}) + \diam(X_{n(y)}) + 1\]
for distinct $x,y \in X$. Both $\diam(X_{n(x)})$ and $\diam(X_{n(y)})$ are non-negative, so $\rho$ satisfies the triangle inequality. For all $x,y \in X$  we have $x,y \in X_{\max\{n(x),n(y)\}}$ so 
\[d(x,y) \leq \diam(X_{\max\{n(x),n(y)\}}) \leq \max\{ \diam(X_{n(x)}), \diam(X_{n(y)}) \} < \rho(x,y)\]
so that $\rho \in M(d)$.
\end{proof}

As a corollary we note that $M(d)$ is non-empty whenever $X$ is countable or when $(X,d)$ is bounded. 

The next example illustrates that even when $T_d$ is non-empty and $X$ is countable, the pair $(T_d,\dinf)$ can fail to be a metric space as there can be $f,g \in T_d$ with $\dinf(f,g) = \infty$.

\begin{example}\label{ex:unboundedTd}
Define 
\[X := \bigcup_{k=1}^\infty \{y_k,z_k\}\]
and let $d$ be the distance function on~$X$ for which $d(y_k,z_k) = k$ for all $k=1,2,\ldots$ and with all other distances~$0$. Let $(\ell_\infty)$ denote the Banach space of infinite sequences with bounded infinity norm. There is a set-valued embedding   $\psi$ of $(X,d)$ into  $(\ell_\infty,d_\infty)$ given by 
\begin{align*}
\psi(y_k) & := \{x \in \ell_\infty: x_k = 0 \} \\
\psi(z_k) & := \{x \in \ell_\infty: x_k = k \} & k = 1,2,\ldots,.
\end{align*}

Define $f:X \rightarrow \Re$ by $f(y_k) = 0$ and $f(z_k) = k$ for all $k=1,2,\ldots$. Let $g:X \rightarrow \Re$ be the complementary function given by $g(y_k) = k$ and $g(z_k) = 0$ for all $k=1,2,\ldots$. Then $f,g \in T_d$ and
\[d_{\infty}(f,g) = \sup\{|f(x) - g(x)| \,:\, x \in X\} = \infty.\]
\qed\end{example}

We have not yet found a satisfying characterization for when a distance space $(X,d)$ has a tight span that is a metric space. One consequence of Theorem~\ref{thm:TdTK} is that if $K$ equals the set of minimal elements of 
 \[ \{ \hx{x}{\rho} : x \in X,\, \rho \in M(d) \},\]
then $(T_d,\dinf)$ is a metric space exactly when $(K,\dinf)$ is. Hence $(T_d,\dinf)$ will be a metric space whenever $X$ is finite, but not necessarily when $X$ is merely countable, as we have seen in Example~\ref{ex:unboundedTd}.

\subsection{Hyperconvexity of smetrics}

As mentioned in the introduction, Smyth and Tsaur \cite{SmythTsaur02} defined a form of hyperconvexity for distance spaces (which they call {\em smetrics}). They call a distance space {\em hyperconvex} if for any indexed collection of closed balls $\{\sB_d(x_\gamma,r_\gamma)\}_{\gamma \in \Gamma}$ satisfying $d(x_\alpha,x_\beta) \leq r_\alpha + r_\beta$ for all $\alpha,\beta \in \Gamma$ it follows that $\bigcap_{\gamma \in \Gamma}  \sB_d(x_\gamma,r_\gamma) \neq \emptyset$.

We find this definition to be too general for our purposes. Indeed, if $(X,d)$ is {\em any} distance space then appending a point at distance zero to all others produces a hyperconvex distance space, according to this definition.

\section{Characterization of subtree distances}
\label{sec:char:subtree:dist}

A metric space $(Z,d_Z)$ is a \emph{real tree} if, for all $a,b \in Z$,
\begin{itemize}
\item[(T1)]
there exists a unique geodesic path from $a$ to $b$, and
\item[(T2)] 
$d_Z(a,\pi(t)) + d_Z(\pi(t),b) = d_Z(a,b)$ for all paths $\pi:[0,d_Z(a,b)] \rightarrow Z$ from $a$ to $b$ and all $0 \leq t \leq d_Z(a,b)$.
\end{itemize}
This definition of real trees  follows~\cite{Dress84}, where it is shown that a metric space $(X,d)$ satisfies~(\ref{eq:4pt}) if and only if $(T_d,\dinf)$ is a real tree~\cite[Thm.~8]{Dress84}. Note that~(T2) can be replaced by various equivalent conditions (see e.g.~\cite{janson2023real}). Moreover, in view of the fact that a real tree is hyperconvex if and only if it is complete (see e.g.~\cite{EspinolaKirk06}), we will restrict in the following to complete real trees.

A \emph{subtree} of a complete real tree $(Z,d_Z)$ is a complete, geodesic,
 non-empty subset of~$Z$. A \emph{subtree representation}
of a distance space $(X,d)$ is a complete real tree $(Z,d_Z)$ together with a family $\{Z_x\}_{x \in X}$ of subtrees of $(Z,d_Z)$ such that
\begin{equation}
\label{eq:subtree:rep}
d(x,y)= \inf\{d_Z(a,b) \,:\, a \in Z_x, \ b \in Z_y\}
\end{equation}
holds for all $x,y \in X$.

In this section, we prove our main result (Theorem~\ref{thm:char:like:hirai}), that Hirai's characterization of distance spaces $(X,d)$ that have a subtree representation holds for general $X$. The proof relies on the following lemmas which establish key properties of $T_d$ when the distance space $(X,d)$ satisfies the extended four-point property~\eqref{eq:extended4pt}.

\begin{lem}
\label{prop:4pt:implies:td:metric}
Let $(X,d)$ be a distance space that satisfies the extended four-point condition~\eqref{eq:extended4pt}. Then $T_d$ is non-empty and $(T_d,\dinf)$ is a metric space.
\end{lem}
\begin{proof}
We first show that $T_d \neq \emptyset$. If $d(u,v) = 0$ for all $u,v \in X$, we have $T_d=\{f\} \neq \emptyset$ with $f(x)=0$ for all $x \in X$. So, assume that there exist $x,y \in X$ with $d(x,y) > 0$. For $n=1,2,\dots$, put $X_n = \{z \in X : \max\{d(x,z),d(y,z)\} \leq \frac{n}{2}\}$. Then $X_1 \subseteq X_2 \subseteq \dots$ is a chain of subsets of $X$ such that $X = \bigcup_{n=1}^{\infty} X_n$. To apply Proposition~\ref{prop:noTd} we need to show that each set $X_n$ is bounded in $(X,d)$. 

Consider $z,w \in X_n$. All of $d(x,z),d(w,x),d(w,y),d(y,z)$ are bounded above by $n/2$. From~\eqref{eq:extended4pt} we have that $d(x,y) + d(w,z)$ is bounded above by one of:
\[ 
\begin{array}{ccc}
d(x,y), & \mbox{ and so } & d(w,z) \leq 0 \\
d(w,x) + d(y,z), & \mbox{ and so } & d(w,z) \leq n/2 + n/2 - d(x,y) \\
d(x,z) + d(w,y), & \mbox{ and so } & d(w,z) \leq n/2 + n/2 - d(x,y) \\
(d(x,y) + d(y,z) + d(z,x))/2, & \mbox{ and so } & d(w,z) \leq n/4 + n/4 - d(x,y)/2 \\
(d(x,y) + d(y,w) + d(w,x))/2 & \mbox{ and so } & d(w,z) \leq n/4+n/4 - d(x,y)/2 \\
(d(x,z) + d(z,w) + d(w,x))/2 & \mbox{ and so } & d(w,z) \leq 2(n/4 + n/4) - 2d(x,y) \\
(d(y,z) + d(z,w) + d(w,y))/2 & \mbox{ and so } & d(w,z) \leq 2(n/4 + n/4) - 2d(x,y).
\end{array}
\]
We note that we cannot have $d(x,y) + d(w,z) \leq d(w,z)$ since $d(x,y) > 0$. In all of these cases, $d(z,w)$, and hence $\diam(X_n)$, is bounded. By Proposition~\ref{prop:noTd}, $T_d \neq \emptyset$, as required.

Next we show that $\dinf(f,g) < \infty$ for all $f,g \in T_d$. Let $f,g \in T_d$. Assume for contradiction that $\dinf(f,g) = \infty$. Without loss of generality we assume that there exists $u \in X$ with $f(u) > 0$. By Theorem~\ref{thm:basicTd}(ii), there exists $v \in X$ such that
\[f(u) + f(v) - \epsilon = d(u,v)\]
for some constant $\epsilon$ with $f(u) > \epsilon \geq 0$. By the assumption that $\dinf(f,g) = \infty$ there exists $w \in X$ such that $f(w) = g(w) + C$ for some constant $C > \max \{g(u) + 3\epsilon,g(v) + 3\epsilon\}$. Moreover, again by Theorem~\ref{thm:basicTd}(ii), there exists some $z \in X$ such that
\[f(w) + f(z) - \epsilon' = d(w,z)\]
for some constant $\epsilon'$ with $\min\{\epsilon,f(w)\} > \epsilon' \geq 0$.

We now consider the cases in~\eqref{eq:extended4pt}. In view of $f(u)>0$, $f(w)>0$ and the choice of $\epsilon$ and $\epsilon'$ we must have $d(u,v) > 0$ and $d(w,z)>0$. Hence, we cannot have $d(u,v) + d(w,z) \leq \max \{d(u,v),d(w,z)\}$. Also, we cannot have
\[d(u,v) + d(w,z) \leq (d(u,v)+d(v,z)+d(u,z))/2\]
since this would imply
\[f(u) + f(v) - \epsilon + 2f(w) + 2f(z) - 2\epsilon' \leq f(v) + f(u) + 2f(z) \]
which simplifies to $2f(w) - \epsilon - 2\epsilon' \leq 0$, contradicting the choice of $C$, $\epsilon$ and~$\epsilon'$. Hence, up to symmetry, there remain only three cases in~\eqref{eq:extended4pt} to consider:
\begin{itemize}
\item
$d(u,v) + d(w,z) \leq d(u,w) + d(v,z)$. Then we have
\[f(u) + f(v) - \epsilon + f(w) + f(z) - \epsilon' \leq d(u,w) + f(v) + f(z),\]
implying $f(w) - \epsilon - \epsilon' \leq d(u,w)$.
\item
$d(u,v) + d(w,z) \leq (d(u,v)+d(v,w)+d(u,w))/2$. Then we have
\[f(u) + f(v) - \epsilon + 2f(w) + 2f(z) - 2\epsilon' \leq d(v,w) + f(u) + f(w),\]
implying $f(w) - \epsilon - 2\epsilon' \leq d(v,w)$.
\item
$d(u,v) + d(w,z) \leq (d(u,w)+d(w,z)+d(u,z))/2$. Then we have
\[2f(u) + 2f(v) - 2\epsilon + f(w) + f(z) - \epsilon' \leq d(u,w) + f(u) + f(z),\]
implying $f(w) - 2\epsilon - \epsilon' \leq d(u,w)$.
\end{itemize}
As the result of this case analysis we obtain
\[f(w) - 3\epsilon \leq \max\{d(u,w),d(v,w)\}.\]
But this implies, in view of the choice of $C$,
\begin{align*}
f(w) - 3\epsilon &\leq \max\{g(u)+g(w),g(v)+g(w)\}\\
&= \max\{g(u)+f(w)-C,g(v)+f(w)-C\} < f(w) - 3\epsilon,
\end{align*}
a contradiction. Hence, $\dinf(f,g) < \infty$, as required.
\end{proof}

We now show that $(T_d,\dinf)$ satisfies the classical  four-point condition. 

\begin{lem}
\label{prop:4pt:implies:td:tree:metric}
Let $(X,d)$ be a distance space that satisfies~\eqref{eq:extended4pt}. Then $(T_d,\dinf)$ is a metric space which satisfies~\eqref{eq:4pt}, the classical four-point condition. 
\end{lem}
\begin{proof}
Let $f_1,f_2,g_1,g_2 \in T_d$.
\begin{align*}
&\dinf(f_1,f_2) + \dinf(g_1,g_2)\\
&= \sup \{d(x,y) - f_1(x) - f_2(y) : x,y \in X\}\\
&\hspace*{4cm} + \sup \{d(w,z) - g_1(w) - g_2(z) : w,z\in X\}\\
& = \sup \{d(x,y) + d(w,z) -  f_1(x) - f_2(y) - g_1(w) - g_2(z) : w,x,y,z \in X\}, 
\end{align*}
where the equality in the third line holds by Theorem~\ref{thm:basicTd}(iii). In the following we will make repeated use of this theorem, the fact that $\dinf$ satisfies the triangle inequality and the fact that $\frac{1}{2}(a+b) \leq \max\{a,b\}$ holds for all $a,b \in \mathbb{R}$.
Fix $w,x,y,z \in X$ and let $\Delta = d(x,y) + d(w,z) -  f_1(x) - f_2(y) - g_1(w) - g_2(z)$.
Assume that $(X,d)$ satisfies~\eqref{eq:extended4pt}. By symmetry, it suffices to consider the following cases:
\begin{itemize}
\item $d(x,y) + d(z,w) \leq d(x,y)$, in which case
\begin{align*}
\Delta & \leq d(x,y) -   f_1(x) - f_2(y) - g_1(w) - g_2(z) \\
& \leq \dinf(f_1,f_2) \leq \frac{1}{2}(\dinf(f_1,g_1) + \dinf(g_1,f_2) + \dinf(f_1,g_2) + \dinf(g_2,f_2)). 
\end{align*}
\item $d(x,y) + d(z,w) \leq d(w,x)+d(y,z)$, in which case
\begin{align*}
\Delta & \leq d(w,x) - f_1(x) - g_1(w) + d(y,z) - f_2(y) - g_2(z) \\
& \leq \dinf(f_1,g_1) + \dinf(f_2,g_2).
\end{align*}
\item $d(x,y) + d(z,w) \leq (d(x,y)+d(y,z)+d(z,x))/2$, in which case
\begin{align*}
    \Delta & \leq (d(x,y)+d(y,z)+d(z,x))/2 -  f_1(x) - f_2(y) - g_1(w) - g_2(z) \\
    & \leq (\dinf(f_1,f_2)+\dinf(f_2,g_2)+\dinf(f_1,g_2))/2.
\end{align*}
\end{itemize}
Hence, for all $\epsilon>0$, there are  $w,x,y,z$ such that 
\begin{align*}
\dinf(f_1,f_2) + \dinf(g_1,g_2)
&\leq \Delta + \epsilon\\
&\leq \max\{ \dinf(f_1,g_2) + \dinf(f_2,g_1) ,\dinf(f_1,g_1) + \dinf(f_2,g_2) \} + \epsilon.
\end{align*}
Taking $\epsilon \rightarrow 0$ we obtain
\[\dinf(f_1,f_2) + \dinf(g_1,g_2) \leq \max\{ \dinf(f_1,g_2) + \dinf(f_2,g_1) ,\dinf(f_1,g_1) + \dinf(f_2,g_2) \},\]
as required.
\end{proof}

With the help of Lemma~\ref{prop:4pt:implies:td:tree:metric} we now establish our main result in this section.

\begin{thm}
\label{thm:char:like:hirai}
Let $(X,d)$ be a distance space. The following are equivalent:
\begin{itemize}
\item[(i)]
$(X,d)$ satisfies the extended four-point condition~\eqref{eq:extended4pt};
\item[(ii)]
$(X,d)$ has a subtree representation;
\item[(iii)]
$(T_d,\dinf)$ is a complete real tree;
\item[(iv)]
$(T_d,\dinf)$ together with the family $\{\kappa(x)\}_{x \in X}$
is a subtree representation of $(X,d)$.
\end{itemize}
\end{thm}

\begin{proof}
$(i) \Rightarrow (iii)$: Assume $(X,d)$ satisfies~\eqref{eq:extended4pt}. Then, by Theorem~\ref{thm:TdHyperconvex} and Lemma~\ref{prop:4pt:implies:td:tree:metric}, $(T_d,\dinf)$ is a hyperconvex metric space that satisfies~\eqref{eq:4pt}. Thus, in view  \cite[Thm.~8]{Dress84}, $(T_d,\dinf)$ is a real tree that is hyperconvex and, therefore, complete. 

$(iii) \Rightarrow (iv)$: Assume that $(T_d,\dinf)$ is a complete real tree.  By Proposition~\ref{prop:kappaProp}, for all $x \in X$, the subset $\kappa(x)$ of $T_d$ is non-empty, complete and geodesic and, thus, a subtree of $(T_d,\dinf)$. Further, in view of Theorem~\ref{thm:kappaEmbed}, we have, for all $x,y \in X$, that 
\[d(x,y) = \inf\{\dinf(f,g):f(x) \in \kappa(x),\, g(y) \in \kappa(y)\}.\]
Hence, $(T_d,\dinf)$ together with the family $\{\kappa(x)\}_{x \in X}$
of subtrees of $(T_d,\dinf)$ is a subtree representation of $(X,d)$.

$(iv) \Rightarrow (ii)$: This is clear.

$(ii) \Rightarrow (i)$: Assume that $(X,d)$ has a subtree representation. 
Then the restriction of $(X,d)$ to any finite subspace also has a subtree representation 
and so $(X,d)$ satisfies~\eqref{eq:extended4pt} by \cite[Thm.~1.2]{Hirai06}.
\end{proof}

Intuitively, Theorem~\ref{thm:char:like:hirai} states that the tight
span of a distance space that satisfies the extended four-point
condition~\eqref{eq:extended4pt} is 1-dimensional. For metric spaces,
this intuition is made more precise in \cite[Sec. 5]{Dress84} by
introducing the concept of the combinatorial dimension of a metric
space and studying metric spaces of bounded combinatorial dimension.
It could be interesting to explore in future work how to extend this
concept to distance spaces.

We close this section establishing that, in the sense made precise in the following theorem, $(T_d,\dinf)$ together with the family $\{\kappa(x)\}_{x \in X}$ is the unique minimal subtree representation of a distance space $(X,d)$ that satisfies~\eqref{eq:extended4pt}. This can be regarded as a generalization of 
the uniqueness of subtree representations in the finite case (cf. \cite{maehara2022optimal}).

\begin{thm}
\label{thm:td:minimal:subtree:representation}
Let $(X,d)$ be a (non-null) distance space that satisfies the extended four-point condition~\eqref{eq:extended4pt}. Then, for any subtree representation of $(X,d)$ consisting of a complete real tree $(Z,d_Z)$ together with the family $\{Z_x\}_{x \in X}$, there exists an isometric embedding $\phi : T_d \rightarrow Z$ of $(T_d,\dinf)$ into $(Z,d_Z)$ such that $\phi(\kappa(x)) \subseteq Z_x$ for all $x \in X$.
\end{thm}

\begin{proof}
By Theorem~\ref{thm:char:like:hirai}, since $(X,d)$ satisfies~\eqref{eq:extended4pt}, the tight span $(T_d,\dinf)$ together with the family $\{\kappa(x)\}_{x \in X}$ make up a subtree representation of $(X,d)$. For all $(x,y) \in X \times X$ with $d(x,y) > 0$ and all $t \in [0,d(x,y)]$, we define  $f_{(x,y,t)}$ as the unique map $f \in T_d$ with $f(x) = t$ and $f(y) = d(x,y)-t$. Intuitively, $f_{(x,y,t)}$ is the unique point on the  geodesic path in $T_d$ between $\kappa(x)$ and $\kappa(y)$ with distance $t$ from $\kappa(x)$ and distance $d(x,y)-t$ from $\kappa(y)$. Let $F_d$ denote the closure in $(T_d,\dinf)$ of the set
\[F_d^{\circ} = \{f_{(x,y,t)}: (x,y) \in X \times X, \ d(x,y) > 0, \ t \in [0,d(x,y)]\}.\]
Hence $F_d \subseteq T_d$. 

We show that $T_d \subseteq F_d$.
Let $f \in T_d$. Since $(X,d)$ is non-null there must exist some $x \in X$ with $f(x) > 0$. Let $f(x) > \epsilon>0$. By Theorem~\ref{thm:basicTd}(ii) there is $y \in X$ such that $f(x) < d(x,y) - f(y) + \epsilon$. Let $t = f(x) - \epsilon$ and consider $g = f_{(x,y,t)} \in F_d^{\circ}$. 
  
There is $h_x \in \kappa(x)$ and $h_y \in \kappa(y)$ such that $g(x) = \dinf(g,h_x) = t$, $g(y) = \dinf(g,h_y) = d(x,y)-t$ and $\dinf(h_x,h_y) = d(x,y)$.
Since $(T_d,\dinf)$ together with the family $\{\kappa(z)\}_{z \in X}$ form a subtree representation of $(X,d)$, we have $\dinf(f,h_x) \leq f(x) + \epsilon$ and $\dinf(f,h_y) \leq f(y) + \epsilon$.

As $(T_d,\dinf)$ satisfies the (classical) four point condition \eqref{eq:4pt} we have 
\begin{align*}
    \dinf(f,g) &\leq \max\{\dinf(f,h_x) + \dinf(g,h_y), \dinf(f,h_y) + \dinf(g,h_x) \} -  \dinf(h_x,h_y)\\
    & \leq \max\{f(x)  + \dinf(g,h_y) + \epsilon, f(y) + \dinf(g,h_x) + \epsilon\} - d(x,y) \\
    & = \max\{f(x)  + d(x,y) - t  + \epsilon, f(y) + t+ \epsilon\} - d(x,y) \\
    & = \max\{2 \epsilon, f(y) + f(x) - d(x,y)\} \leq \max\{2 \epsilon, \epsilon\} \leq 2 \epsilon. 
\end{align*}
Hence $F_d^{\circ}$ is dense in $T_d$.

In an analogous way we define, for all $(x,y) \in X \times X$ with $d(x,y) > 0$ and all $t \in [0,d(x,y)]$, the unique point $p_{(x,y,t)}$ in $Z$ that lies on the geodesic path between $Z_x$ and $Z_y$ such that
\[d(x,y) = \inf \{d_Z(p_{(x,y,t)},q) : q \in Z_x\} + \inf \{d_Z(p_{(x,y,t)},q) : q \in Z_y\}.\]
Let $G_d$ denote the completion in $(Z,d_Z)$ of the set
\[G_d^{\circ} = \{p_{(x,y,t)}: (x,y) \in X \times X, \ d(x,y) > 0, \ t \in [0,d(x,y)]\}.\]
By construction, we have $G_d \subseteq Z$ but equality need not hold.
Thus, it suffices to give an isometry $\phi : F_d \rightarrow G_d$ between $(F_d,\dinf)$ and $(G_d,d_Z)$. Let $(x,y),(a,b) \in X \times X$ with $d(x,y) > 0$ and $d(a,b) > 0$. Consider the restriction of $d$ to the finite set $M=\{x,y,a,b\}$. Then, in view of \cite[Thm.~2]{maehara2022optimal}, we have
\[\dinf(f_{(x,y,t)},f_{(a,b,t')}) = d_Z(p_{(x,y,t)},p_{(a,b,t')})\]
for all $t \in [0,d(x,y)]$ and all $t' \in [0,d(a,b)]$. Hence, putting $\phi(f_{(x,y,t)}) = p_{(x,y,t)}$ for all $(x,y) \in X \times X$ with $d(x,y) > 0$ and all $t \in [0,d(x,y)]$ yields an isometry between $(F_d^{\circ},\dinf)$ and $(G_d^{\circ},d_Z)$ that can be extended to an isometry between $(F_d,\dinf)$ and $(G_d,d_Z)$.

It remains to show that $\phi(\kappa(x)) \subseteq Z_x$ for all $x \in X$. Assume for a contradiction that there exist $(a,b) \in X \times X$ with $d(a,b) > 0$, $t \in [0,d(a,b)]$ and $x \in X$ such that $f_{(a,b,t)}(x)=0$ but $p_{(a,b,t)} \not \in Z_x$. Then, without loss of generality, we have
\[d(a,x) > d_Z(p_{(a,b,0)},p_{(a,b,t)}) = \dinf(f_{(a,b,0)},f_{(a,b,t)}) \geq d(a,x),\]
a contradiction. Thus, we have $\phi(\kappa(x) \cap F_d^{\circ}) \subseteq Z_x$ for all $x \in X$, implying, in view of the fact that both $\kappa(x)$ and $Z_x$ are complete, that $\phi(\kappa(x)) \subseteq Z_x$, as required.
\end{proof}

The following example illustrates that we cannot remove the condition that there exist $x,y \in X$ with $d(x,y) > 0$ from Theorem~\ref{thm:td:minimal:subtree:representation}.

\begin{example}
\label{ex:distance:all:zero:minimal}
Let $X=\{a_n : n \in \mathbb{N}\}$ and consider the distance $d$ on $X$ with $d(a_i,a_j)=0$ for all $i,j \in \mathbb{N}$. 
Then we have $T_d=\{f\}$ with $f(a_i) = 0$ for all $i \in \mathbb{N}$ and $\kappa(a_i)=\{f\}$ for all $i \in \mathbb{N}$. The real tree with $Z=\{x \in \mathbb{R} : x \geq 0\}$ and $d_Z(x,y) = |x-y|$ together with the family $\{Z_{a_i}\}_{i \in \mathbb{N}}$ defined by $Z_{a_i} = \{x \in \mathbb{R} : x \geq i\}$ is a subtree representation of $(X,d)$. There exists, however, no point $z \in Z$ with $z \in Z_{a_i}$ for all $i \in \mathbb{N}$, implying that there exists no embedding $\phi : T_d \rightarrow Z$ with $\phi(\kappa(a_i)) \subseteq Z_{a_i}$ for all $i \in \mathbb{N}$.
\qed\end{example}

\section{Diversities}
\label{sec:Diversities}

In this section we study the relationship between the tight span of diversities and distances. We begin by recalling some relevant facts about diversities.

Let $\Pf(X)$ denote the set of all finite subsets of $X$ and $\Pfs(X)$ the non-empty finite subsets of $X$ . A \emph{diversity} on $X$ is a pair $(X,\delta)$ such that $\delta : \Pf(X) \rightarrow \mathbb{R}_{\geq 0}$ satisfies
\begin{itemize}
\item[(i)]
$\delta(A) = 0$ if and only if $|A| \leq 1$, and
\item[(ii)]
$\delta(A \cup C) \leq \delta(A \cup B) + \delta(B \cup C)$ if $B \neq \emptyset$.
\end{itemize}
for all $A,B,C \in \Pf(X)$.
Note that the term diversity is used {\em for both the pair $(X,\delta)$ and the map $\delta$}. One direct consequence of the definition of a diversity is that the restriction of $\delta$ to pairs is a metric, called the metric \emph{induced} by~$\delta$ on~$X$.

We recall some examples of diversities from \cite[Sec.~1]{BryantTupper12} that we will return to later on.

\begin{example}
\begin{enumerate}
\item[(i)]
Any metric space $(X,d)$ with $d(x,y) > 0$ if $x \neq y$ gives rise to the \emph{diameter diversity} $\delta = \delta_d$ on~$X$ defined by putting $\delta(\emptyset)=0$ and $\delta(A) = \max\{d(x,y) : x,y \in A\}$ for all $A \in \Pf(X) \setminus \{\emptyset\}$. Clearly, the metric induced by the diameter diversity on~$X$ is just $d$.
\item[(ii)]
For $X=\mathbb{R}^n$ the $\ell_1$-\emph{diversity}~$\delta$ on~$X$ is defined by putting $\delta(\emptyset)=0$ and
\[\delta(A) = \sum_{i=1}^n \max \{|x_i - y_i| : x,y \in A\}\]
for all $A \in \Pfs(X)$. The diversity equals the sum of the lengths in each direction of the smallest axis-aligned rectangle containing $A$. The induced metric of this diversity is the $\ell_1$-metric on $\Re^n$. 
\item[(iii)]
Let $(Z,d_Z)$ be a complete real tree. The intuition behind the  {\em real-tree} diversity $(Z,\delta_Z)$ for $(Z,d_Z)$ is that for each $A \in \Pfs(Z)$, the diversity  $\delta_Z(A)$ of $A$ equals the length of the shortest subtree connecting $A$ \cite{Faith92}. 
More formally, let $\mu$ be the one-dimensional Hausdorff measure on $(Z,d_Z)$ \cite{Edgar08}. We note that $\mu$ is defined on all Borel sets, it is monotone, and it is additive on disjoint sets, and that if $[a,b]$ denotes the points on the path from $a$ to $b$ in $Z$ then $\mu([a,b]) = d(a,b)$. Let $\conv(A) = \bigcup_{a,b \in A} [a,b]$ denote the convex hull of $A$, which is closed when $A$ is finite. We then have
\[\delta_Z(A) = \mu\left( \conv(A) \right).\]
A diversity $\delta$ on a set $X$ is a \emph{phylogenetic diversity} if there exists a complete real tree $(Z,d_Z)$ with $X \subseteq Z$ and $\delta$ is the restriction of $\delta_{Z}$ to~$\Pf(X)$.
\end{enumerate}
\end{example}

Let $(X,\delta)$ be a diversity. For $f,g \in \Re^{\Pf(X)}$ we write $f \preceq g$ if $f(A) \leq g(A)$ for all $A \in \Pf(X)$. Put
\begin{align}
P_{\delta} &:= \{f \in \mathbb{R}^{\Pf(X)} : f(\emptyset) = 0 \ \text{and} \ \sum_{C \in \mathcal{C}} f(C) \geq \delta(\bigcup_{C \in \mathcal{C}}) \ \text{for all finite } \ \mathcal{C} \subseteq \Pfs(X)\},\notag\\
T_{\delta} &:= \{f \in P_{\delta} : f \ \text{is minimal with respect to} \ \preceq\}.
\label{eq:p:delta:t:delta}
\end{align}
$T_{\delta}$ is called the \emph{tight span} of $(X,\delta)$. By \cite[Thm.~2.7]{BryantTupper12}, the map $\delta_{T} : \Pf(T_{\delta}) \rightarrow \mathbb{R}$ defined by putting $\delta_{T}(\emptyset) = 0$ and, for $F \neq \emptyset$,
\begin{equation}
\delta_{T}(F) = \sup_{\substack{\mathcal{C} \subseteq \Pf(X)\\ \mathcal{C} \ \text{finite}}} \left \{\delta(\bigcup_{C \in \mathcal{C}}) - \sum_{C \in \mathcal{C}} \inf_{f \in F} f(C) \right \} 
\end{equation}
is a diversity on~$T_{\delta}$. Defining, for all $x \in X$, the map $g_x : \Pf(X) \rightarrow \mathbb{R}_{\geq 0}$ by putting $g_x(A) := \delta(A \cup \{x\})$, we have, by \cite[Thm.~2.8]{BryantTupper12}, that $g_x \in T_{\delta}$ and
\begin{equation}
\label{eq:embed:div:t:delta}
\delta(C) = \delta_{T}(\{g_c : c \in C\})
\end{equation}
for all $C \in \Pf(X)$. That is, the diversity $(X,\delta)$ is embedded in the diversity $(T_{\delta},\delta_{T})$. In addition, it follows from \cite[Lem.~2.6]{BryantTupper12} that the metric induced by $\delta_{T}$ on $T_{\delta}$ coincides with $\dinf$.

Given a diversity $(X,\delta)$, we define a distance $D_{\delta}$ on $\Pf(X)$
by putting
\[D_\delta(A,B) := \max\{\delta(A \cup B)-\delta(A) - \delta(B) ,0 \} \]
(cf. \cite{HerrmannMoulton12}).
As we shall see, this construction provides us with a link between the tight span of diversities and the tight span of distance spaces. In particular,
we show in this section that there is an embedding of the tight span of the diversity $(X,\delta)$ into the tight span of the distance space $(\Pf(X),D_\delta)$, and we characterize when this embedding is bijective.

Note that, for any diversity $(X,\delta)$, both $T_\delta$ and the tight span $T_D$ of the distance space $(\Pf(X),D)=(\Pf(X),D_\delta)$ consist of maps from $\Pf(X)$ to $\Re$. In addition, also by definition, $\delta$ is itself a map from $\Pf(X)$ to $\Re$. With this in mind, the next theorem describes how $T_{\delta}$ and $T_D$ are related (for finite $X$ this relationship is stated in \cite[p.~2517]{HerrmannMoulton12}).

\begin{thm}
\label{thm:TdeltaEmbed}
Let $(X,\delta)$ be a diversity, $D = D_{\delta}$ and $f \in \mathbb{R}^{\Pf(X)}$.
If $f \in T_\delta$, then $f - \delta \in T_D$.
\end{thm}
\begin{proof}
Suppose that $f \in T_\delta$ and let $g = f - \delta$. Then, for all $A,B \in \Pf(X)$, we have 
\[g(A) + g(B) =  f(A) -\delta(A) + f(B) - \delta(B) \geq 0\]
and
\[ g(A) + g(B) = f(A) + f(B) - \delta(A) - \delta(B) \geq \delta(A \cup B) - \delta(A) - \delta(B).\]
Hence,
\[g(A) + g(B) \geq \max\{\delta(A \cup B)-\delta(A) - \delta(B) ,0 \} = D(A,B)\]
and, therefore, $g \in P_D$. 

To show that $g \in T_{D}$, consider $A \in \Pf(X)$. If $f(A) = \delta(A)$ we have $g(A) = 0 = D(A,A)$. Otherwise, by \cite[Prop.~2.4~(5)]{BryantTupper12} there is some $B \in \Pf(X)$ such that 
\[0 \leq f(A) = \delta(A \cup B) - f(B).\]
Thus  $\delta(A \cup B) - \delta(A) - \delta(B) \geq 0$ and 
\[g(A) = \delta(A \cup B) - \delta(A) - \delta(B) - g(B) = D(A,B) - g(B),\]
implying $g(A) = \sup \{D(A,B) - g(B) : B \in \Pf(X)\}$. Hence, by Theorem~\ref{thm:basicTd} (ii), $g \in T_D$.
\end{proof}

As an immediate corollary of Theorem~\ref{thm:TdeltaEmbed} we obtain that $T_{\delta}$ is isometrically embedded into $T_D$.

\begin{cor}
\label{cor:delta:ddelta:isometric:embedding}
Let $(X,\delta)$ be a diversity, $D = D_{\delta}$ and $f_1,f_2 \in T_\delta$. Put $g_1 = f_1 - \delta$ and $g_2 = f_2 - \delta$. Then
\[\dinf(f_1,f_2) = \dinf(g_1,g_2) < \infty.\]
\end{cor}

\begin{proof}
We have $\dinf(f_1,f_2) = \delta_{T}(\{f_1,f_2\})$ implying that $\dinf(f_1,f_2) < \infty$. Moreover, we clearly have $\dinf(f_1,f_2) = \dinf(g_1,g_2)$ by construction of $g_1$ and $g_2$.
\end{proof}

\begin{figure}
\centering
\includegraphics[scale=1.0]{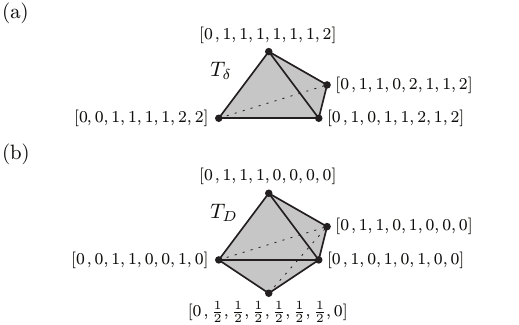}
\caption{\small (a) The tight span of the diversity $(X,\delta)$ in Example~\ref{ex:relationship:t:d:delta} with $X=\{x,y,z\}$. The three vertices at the base of the pyramid correspond to the elements in~$X$. (b) The tight span of the associated distance space $(\Pf(X),D)$ with $D=D_{\delta}$.}
\label{fig:ex:rel:t:d:delta}
\end{figure}

\begin{example}
\label{ex:relationship:t:d:delta}
Consider $X=\{x,y,z\}$ and the diversity $\delta$ on $X$ defined by
$\delta(\emptyset)=0$ and $\delta(A) = |A|-1$ for all $A \in \Pf(X) \setminus \{\emptyset\}$. We write
\[f=\Big[f(\emptyset),f(\{x\}),f(\{y\}),f(\{z\}),f(\{x,y\}),f(\{x,z\}),f(\{y,z\}),f(\{x,y,z\})\Big]\]
for a map $f \in \mathbb{R}^{\Pf(X)}$. The tight span $T_{\delta}$ is the triangular pyramid shown in Figure~\ref{fig:ex:rel:t:d:delta}(a). In contrast, the tight span $T_D$ for the distance $D=D_{\delta}$ is the triangular bi-pyramid shown in Figure~\ref{fig:ex:rel:t:d:delta}(b). Thus, $T_D$ contains a translate of $T_{\delta}$ as described in Theorem~\ref{thm:TdeltaEmbed} but it is not in bijection with $T_{\delta}$.
\qed\end{example}

As can be seen from Example~\ref{ex:relationship:t:d:delta}, the embedding of $T_{\delta}$ into $T_D$ via Theorem~\ref{thm:TdeltaEmbed} need not be a bijection. We shall see in the remainder of this section that there are diversities for which the embedding is a bijection. Theorem~\ref{thm:nice} below gives a characterization of these diversities. This characterization involves the following relaxation of the conditions that define the sets $P_{\delta}$ and $T_{\delta}$ in~\eqref{eq:p:delta:t:delta} for a diversity $(X,\delta)$:
\begin{align}
P^{(2)}_{\delta} &= \{f \in \mathbb{R}^{\Pf(X)} : f(\emptyset) = 0 \ \text{and} \ f(A) + f(B) \geq \delta(A \cup B) \  \text{for all} \ A,B \in \Pf(X)\},\notag\\
T^{(2)}_{\delta} &= \{f \in P^{(2)}_{\delta} : f \ \text{minimal with respect to} \ \preceq\}.
\label{eq:p:2:delta:t:2:delta}
\end{align}

Recall that for the distance space $(\Pf(X),D)$ and $A \in \Pf(X)$ we have
\[\kappa(A) = \{g \in T_D:g(A) = 0\}.\]

\begin{thm}
\label{thm:nice}
Let $(X,\delta)$ be a diversity and $D=D_{\delta}$.  Then the following are equivalent:
\begin{enumerate}
\item[(i)] $P_\delta = P^{(2)}_\delta$;
\item[(ii)] $T_\delta = T^{(2)}_\delta$;
\item[(iii)] For all $f \in \mathbb{R}^{\Pf(X)}$, $f \in T_\delta$ if and only if $f - \delta \in T_D$;
\item[(iv)]
The metric space $(T_{\delta},\dinf)$ is hyperconvex and $\kappa(A) + \delta \subseteq T_{\delta}$ for all $A \in \Pf(X) \setminus \{\emptyset,X\}$.
\end{enumerate}
\end{thm}
\begin{proof}
$(i) \Rightarrow (ii)$: Clearly, if  $P_\delta = P^{(2)}_\delta$ then $T_\delta = T^{(2)}_\delta$. 

$(ii) \Rightarrow (iii)$: Suppose $T_\delta = T^{(2)}_\delta$ and consider $f \in \mathbb{R}^{\Pf(X)}$. Put $g = f - \delta$. In view of Theorem~\ref{thm:TdeltaEmbed}, we have $g \in T_D$ if $f \in T_{\delta}$. It remains to show that $g \in T_D$ implies $f \in T_{\delta}$. So, assume $g \in T_D$. Then, for all $A \in \Pf(X)$,
\[g(A) + g(A) \geq  D(A,A) \geq 0\]
so that $g$ is non-negative, and, in view of Theorem~\ref{thm:basicTd}(ii) and the definition of~$D$, 
\[g(\emptyset)  = \sup\{ D(\emptyset,A) - g(A) : A \in \Pf(X) \} \leq 0\]
so that $g(\emptyset) = 0$ and, thus, $f(\emptyset) = 0$. 
Moreover, for all $A,B \in \Pf(X)$ (not necessarily distinct) we have 
\[f(A) - \delta(A) + f(B) - \delta(B) = g(A) + g(B) \geq D(A,B) \geq \delta(A \cup B) - \delta(A) - \delta(B)\]
so that $f(A) + f(B) \geq \delta(A \cup B)$. Hence $f \in P^{(2)}_\delta$. 

To show that $f \in T_\delta$, consider $A \in \Pf(X)$. If $g(A) = 0$ then $f(A) - \delta(A) = 0 = -f(\emptyset)$, so $f(A) + f(\emptyset) = \delta(A)$. Otherwise, since $g \in T_D$, we have, again in view of Theorem~\ref{thm:basicTd}(ii), 
\[g(A) = \sup \{D(A,B) - g(B) : B \in \Pf(X)\} > 0,\]
implying that it suffices to take the supremum over those $B \in \Pf(X)$ with $0 < D(A,B) = \delta(A \cup B) - \delta(A) - \delta(B)$.
Thus, by definition of $g$, we have
\[f(A) = \sup \{\delta(A \cup B) - f(B) : B \in \Pf(X)\}.\]
Hence, $f$ is minimal in $P^{(2)}_\delta$ with respect to pointwise comparison $\preceq$, implying, by definition, $f \in T^{(2)}_\delta = T_\delta$. 

$(iii) \Rightarrow (i)$: By definition, we have $P_\delta \subseteq P^{(2)}_\delta$. To show that also $P^{(2)}_\delta \subseteq P_\delta$, consider $f \in P^{(2)}_\delta$.  Then, by definition, we have $f(\emptyset)=0$, $f(A) \geq \delta(A)$, $f(B) \geq \delta(B)$ and 
\[f(A) + f(B) \geq \delta(A \cup B)\]
for all $A,B \in \Pf(X)$.
Let $g = f - \delta$. Then $g(A) \geq 0$, $g(B) \geq 0$ and 
\[g(A) + g(B) \geq \max\{\delta(A \cup B) - \delta(A) - \delta(B) , 0\} = D(A,B),\]
implying $g \in P_D$. Hence, by Theorem~\ref{thm:basicTd}(i), there exists
$g' \in T_D$ such that $g' \preceq g$. In view of (iii), this implies $f' = g' + \delta \in T_\delta$. Hence, also for all finite $\mathcal{C} \subseteq \Pfs(X)$ with $|\mathcal{C}| \geq 3$ we have
\[\delta(\bigcup_{C \in \mathcal{C}}C)
\leq \sum_{C \in \mathcal{C}} f'(C)
= \sum_{C \in \mathcal{C}} (g'(C) + \delta(C))
\leq \sum_{C \in \mathcal{C}} (g(C) + \delta(C))
= \sum_{C \in \mathcal{C}} f(C),\]
implying $f \in P_\delta$.

$(iii) \Rightarrow (iv)$: 
Suppose that~(iii) holds. Then, by Theorem~\ref{thm:TdeltaEmbed} and Corollary~\ref{cor:delta:ddelta:isometric:embedding}, the metric space $(T_{\delta},\dinf)$ is isometric to the metric space $(T_D,\dinf)$. In particular, $T_D \neq \emptyset$ and $\dinf(f,g) < \infty$ for all $f,g \in T_D$. Thus, by Theorem~\ref{thm:TdHyperconvex}, $(T_D,\dinf)$ and, therefore, $(T_{\delta},\dinf)$ as well are hyperconvex metric spaces. Moreover, by the assumption that $T_D  = T_{\delta} - \delta$, we clearly have $\kappa(A) + \delta \subseteq T_{\delta}$ for all $A \in \Pf(X) \setminus \{\emptyset,X\}$, as required.

$(iv)\Rightarrow (iii)$:  assume that $(T_{\delta},\dinf)$ is a hyperconvex metric space and 
\[\kappa(A) + \delta = \{g \in T_D:g(A) = 0\} + \delta  \subseteq T_{\delta}\]
for all $A \in \Pf(X)$. 
By Theorem~\ref{thm:TdeltaEmbed} and Corollary~\ref{cor:delta:ddelta:isometric:embedding} there exists an isometric embedding $\pi : T_{\delta} \rightarrow T_D$, with $\pi(f) = f - \delta$, and the identity map $\phi : T_{\delta} \rightarrow T_{\delta}$ clearly is non-expansive. Hence, by the assumption that $(T_{\delta},\dinf)$ is a hyperconvex and, therefore, injective metric space, there exists a non-expansive map $\psi : T_D \rightarrow T_{\delta}$ such that $\psi(\pi(f)) = f$ for all $f \in T_{\delta}$. 

Now, to show that~(iii) holds, it suffices to show that $\psi$ is a bijection. Consider $g \in T_D$ and put $f = \psi(g)$. Note that $D(\emptyset,B) = 0$ for all $B \in \Pf(X)$ and if $X$ is finite we further have $D(X,B) = 0$. Hence, by Theorem~\ref{thm:basicTd}(ii), we have 
\[g(\emptyset) = \sup\{D(\emptyset,B) - D(B) : B \in \Pf(X)\} = 0\]
so that 
\[g(\emptyset) = f(\emptyset) - \delta(\emptyset)\]
and, when $X$ is finite, $g(X) = 0$ and 
\[g(X) =  f(X) - \delta(X).\]
Moreover, since $\psi$ is non-expansive, we have, in view of Theorem~\ref{thm:kappaEmbed},
\begin{align*}
g(A) &= \inf\{\dinf(g,h): h \in \kappa(A)\}\\
& \geq \inf \{\dinf(\psi(g),\psi(h)): h \in \kappa(A)\}& \mbox{(as $\psi$ is non-expansive)}\\
&=  \inf\{\dinf(f,h + \delta): h \in \kappa(A) \} & \mbox{(as $\psi(h) = h + \delta$)}\\
&=  \inf\{\dinf(f-\delta,h): h \in \kappa(A) \}\\
&= f(A) - \delta(A)
\end{align*}
for all $A \in \Pf(X) \setminus \{\emptyset,X\}$. Since, by the definition of $T_D$, $g$ is minimal with respect to $\preceq$ in $T_D$ and $f - \delta \in T_D$, we must have $g = f - \delta = \pi(f)$. Hence, $\psi$ is a bijection.
\end{proof}

As an immediate consequence of Theorem~\ref{thm:nice} we determine two important classes of diversities for which $T_D = T_\delta - \delta$. 

\begin{cor}
For the following diversities $(X,\delta)$ we have $T_D = T_\delta - \delta$ with $D=D_{\delta}$:
\begin{enumerate}
\item 
The diameter diversity $\delta=\delta_d$ on $X$ associated to any metric space $(X,d)$ with $d(x,y) > 0$ if $x \neq y$.
\item 
The $\ell_1$-diversity $\delta$ on $X=\mathbb{R}^n$.
\end{enumerate}
\end{cor}

\begin{proof}
In view of Theorem~\ref{thm:nice} and the fact that $P_\delta \subseteq P^{(2)}_\delta$, it suffices to show that $P_\delta^{(2)} \subseteq P_\delta$.
\begin{enumerate}
\item
Let $f \in P_\delta^{(2)}$ and $\mathcal{C} \subseteq \Pfs(X)$ finite with $|\mathcal{C}| \geq 3$. Then we have 
\begin{align*}
\delta\left(\bigcup_{C \in \mathcal{C}} C\right) & = \max\{ \delta(C \cup C') : C,C' \in \mathcal{C}, C \neq C'\}\\
& \leq  \max\{ f(C) + f(C') : C,C' \in \mathcal{C}, C \neq C'\} \leq \sum_{C \in \mathcal{C}} f(C),
\end{align*}
so $f \in P_\delta$, as required.
\item 
For $X$ a finite subset of $\mathbb{R}^n$ it was shown in \cite[Thm.~7.4]{HerrmannMoulton12} that $P_{\delta'}^{(2)} = P_{\delta'}$ for the restriction $\delta'$ of $\delta$ to $\Pf(X)$. To see that this also holds for $\delta$, we use a similar argument as for diameter diversities above: Let $f \in P_\delta^{(2)}$ and $\mathcal{C} \subseteq \Pfs(X)$ finite with $|\mathcal{C}| \geq 3$. Then by applying Proposition~5 of \cite{BryantTupper14} to every possible ordering of $\mathcal{C}$ we have
\begin{align*}
\delta(\bigcup_{C \in \mathcal{C}} C)
&\leq \frac{1}{|\mathcal{C}|-1} \cdot \sum_{\substack{\{C,C'\} \subseteq \mathcal{C}\\C \neq C'}} \delta(C \cup C')\\
&\leq \frac{1}{|\mathcal{C}|-1} \cdot \sum_{\substack{\{C,C'\} \subseteq \mathcal{C}\\C \neq C'}} f(C) + f(C') =\sum_{C \in \mathcal{C}} f(C),
\end{align*}
implying $f \in P_\delta$, as required.
\end{enumerate}
\end{proof}

\section{Arboreal Diversities}
\label{sec:arborel:div}

One of the main motivations for the development of 
diversity theory was the study of phylogenetic diversities.
These have several interesting properties, one of
which being that they have tight spans which induce metric spaces that are real trees \cite[Thm.~5.7]{BryantTupper12}.
More generally, we shall say that a  diversity $(X,\delta)$ is \emph{arboreal} 
if its tight span $(T_{\delta},\dinf)$ has an induced metric space that is a real tree. 
In this section, we shall give a characterization for such diversities. Note that, 
as can be seen from the following example, there exist arboreal diversities $(X,\delta)$ such that $\delta$ is \emph{not} a phylogenetic diversity. These provide a counter-example to Theorem 7.3 of \cite{HerrmannMoulton12}.

\begin{figure}
\centering
\includegraphics[scale=1.0]{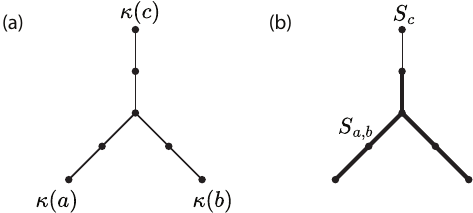}
\caption{\small (a) The tight span of the diversity $(X,\delta)$ in Example~\ref{ex:arboreal:not:phylogenetic} with $X=\{a,b,c\}$, which is a real tree with three leaves. Each leaf corresponds to an element in~$X$. Every edge in the tree has length~1. (b) The subtrees representing the subsets $\{c\}$ (a single leaf) and $\{a,b\}$ (drawn bold) of $X$ in the subtree representation of $(\Pf(X),D)$ with $D=D_{\delta}$.}
\label{fig:ex:arboreal:not:pd}
\end{figure}

\begin{example}
\label{ex:arboreal:not:phylogenetic}
Consider the diversity $\delta$ on $X=\{a,b,c\}$ defined as $\delta(\emptyset) = \delta(\{a\}) = \delta(\{b\}) = \delta(\{c\}) = 0$,
$\delta(\{a,b\}) = \delta(\{a,c\}) = \delta(\{b,c\}) = 4$ and
$\delta(\{a,b,c\}) = 5$. The tight span $T_{\delta}$ with the metric $\dinf$ is a real tree and shown in Figure~\ref{fig:ex:arboreal:not:pd}(a). $\delta$ is, however, not a phylogenetic diversity because then $\delta(X)$ would need to equal~6, the total length of the real tree in Figure~\ref{fig:ex:arboreal:not:pd}(a). Even so, the distance space $(\Pf(X),D=D_{\delta})$ has a subtree representation and, thus, satisfies~\eqref{eq:extended4pt}. In Figure~\ref{fig:ex:arboreal:not:pd}(b) the subtrees $S_{\{a,b\}}$ and $S_{\{c\}}$ obtained from the real tree in Figure~\ref{fig:ex:arboreal:not:pd}(a) are shown for the two subsets $\{a,b\}$ and $\{c\}$. It can be seen that $D(\{a,b\},\{c\}) = \max\{\delta(\{a,b,c\}) - \delta(\{a,b\}) - \delta(\{c\}),0\} = 1$ equals the distance between the two subtrees.
\qed\end{example}

As we shall see, a key tool to establish a characterization of arboreal diversities is hyperconvexity. A diversity $(X,\delta)$ is \emph{hyperconvex} if, for any collection $\{A_\gamma\}_{\gamma \in \Gamma}$ of subsets in $\Pf(X)$ and non-negative real numbers $\{r_\gamma\}_{\gamma \in \Gamma}$ such that
\begin{equation}
\label{eq:hyperconvex:div}
\delta \left ( \bigcup_{\gamma \in \Gamma'} A_\gamma \right ) \leq \sum_{\gamma \in \Gamma'} r_{\gamma}
\end{equation}
for all finite non-empty subsets $\Gamma' \subseteq \Gamma$, there exists $z \in X$ with $\delta(A_\gamma \cup \{z\}) \leq r_\gamma$ for all $\gamma \in \Gamma$.
It is shown in \cite[Thm.~3.6]{BryantTupper12} that, for any diversity $(X,\delta)$, the tight span $(T_{\delta},\delta_{T})$ is a hyperconvex diversity. However, as can be seen in Example~\ref{ex:relationship:t:d:delta} by considering the set of three points $\{g_x,g_y,g_z\}$ and $r_x=r_y=r_z=\frac{1}{2}$, the metric space $(T_{\delta},\dinf)$ need not be hyperconvex (see also \cite{EspinolaPiatek14}). 

We first establish some basic properties of the metric space $(T_{\delta},\dinf)$.

\begin{lem}
\label{lem:div:hyperconvex:path:connected}
 The induced metric space  $(X,d)$  of a hyperconvex diversity $(X,\delta)$ is complete and geodesic.
\end{lem}
\begin{proof}
That $(X,d)$ is complete was established in Proposition 3.10 of \cite{EspinolaPiatek14}.

We now show that $(X,d)$ is geodesic. 
Let $x,y \in X$ and suppose $\ell = \delta(\{x,y\}) = d(x,y)>0$. Let $Y_1 = \{0,\ell\}$ and let $\delta_1$ be the unique diversity on $Y_1$ with $\delta_1(\{0,\ell\}) = \ell$. Let $Y_2$ be the real interval $[0,\ell]$ and let $\delta_2$ be the diameter diversity on $Y_2$ with the standard metric, that is, $\delta_2(A) = \max(|a_1 - a_2|:a_1,a_2 \in A\}$ for all finite $A \subseteq [0,\ell]$.  Let $\pi:Y_1 \rightarrow Y_2$ be the embedding with $\pi(0) = 0$ and $\pi(\ell) = \ell$, and let $\phi:Y_1 \rightarrow X$ be the map with $\phi(0) = x$ and $\phi(\ell) = y$. By Theorem 3.3 of \cite{BryantTupper12} there is a non-expansive map $\psi$ from $Y_2 = [0,\ell]$ to $X$ such that $\psi(0) = x$ and $\psi(\ell) = y$. It follows that $\psi$ is a geodesic. 
\end{proof}

The next lemma gives a sufficient condition for a diversity $(X,\delta)$ to be arboreal.

\begin{lem}
\label{lem:arboreal:sufficient}
Let $(X,\delta)$ be a diversity such that the associated distance space $(\Pf(X),D=D_\delta)$ satisfies the extended four-point condition~\eqref{eq:extended4pt}. Then $(X,\delta)$ is arboreal.
\end{lem}

\begin{proof}
By \cite[Thm.~3.6]{BryantTupper12}, the diversity $(T_{\delta},\delta_{T})$ is hyperconvex and so, by Lemma~\ref{lem:div:hyperconvex:path:connected}, the metric space $(T_{\delta},\dinf)$ is geodesic.
As $(\Pf(X),D)$ satisfies~\eqref{eq:extended4pt}, $(T_D,\dinf)$ is a real tree, in view of Theorem~\ref{thm:char:like:hirai}. Thus, by Theorem~\ref{thm:TdeltaEmbed} and Corollary~\ref{cor:delta:ddelta:isometric:embedding}, $(T_{\delta},\dinf)$ is isometric to a non-empty, path-connected subset of the real tree $(T_D,\dinf)$. This implies that the induced metric space $(T_{\delta},\dinf)$ is a real tree and, thus, $(X,\delta)$ is arboreal.
\end{proof}

In the remainder of this section, we will show that the converse of Lemma~\ref{lem:arboreal:sufficient} holds. By Theorem~\ref{thm:TdeltaEmbed} there is an isometric embedding of $(T_\delta, \dinf)$ into $(T_D,\dinf)$ however this embedding need not be bijective. Instead we establish several new results about hyperconvex diversities and subsequently 
construct such a subtree representation of $(\Pf(X),D)$ directly from $(T_\delta,\dinf)$.

For a diversity $(X,\delta)$ and $A \in \Pfs(X)$ 
we define
\[\sB_A := \{x \in X: \delta(A \cup \{x\}) \leq \delta(A)\}.\]
Note that $A \subseteq \sB_A$, so $\sB_A$ is non-empty. 

\begin{prop}
\label{prop:ba:hyperconvex}
Let $(X,\delta)$ be a hyperconvex diversity and $A \in \Pfs(X)$. Then $(\sB_A,\delta)$ is also a hyperconvex diversity.
\end{prop}
\begin{proof}
This is a direct consequence of Proposition~3.8 in \cite{EspinolaPiatek14}, letting $Z = A$ and $r(Z') = \delta(A)$ for all $Z' \subseteq A$, so that $Y = \sB_A$.
\end{proof}

In view of Lemma~\ref{lem:div:hyperconvex:path:connected}, we immediately obtain the following consequence of Proposition~\ref{prop:ba:hyperconvex}:

\begin{cor}
\label{cor:ba:path:connected}
Let $(X,\delta)$ be a hyperconvex diversity with induced metric $(X,d)$ and suppose $A \in \Pfs(X)$. The restriction of $(X,d)$ to $\sB_A$ is complete and geodesic.
\end{cor}

The next proposition links, for any hyperconvex diversity $(X,\delta)$, the induced distances between the sets $\sB_A$ for $A \in \Pfs(X)$ and the distance space $(\Pf(X),D_\delta)$.

\begin{prop}
\label{prop:embedD}
Let $(X,\delta)$ be a hyperconvex diversity and $d=d_{\delta}$ be the metric induced by $\delta$ on~$X$. Then, for all $A,B \in \Pfs(X)$,
\[D_\delta(A,B) = \inf\{d(x,y): x \in \sB_A,\, y \in \sB_B\}.\]
\end{prop}

\begin{proof}
Let $\ell = \inf\{d(x,y)  \,:\, x \in \sB_A,\, y \in \sB_B\}$.

For all $x \in \sB_A$ and $y \in \sB_B$ we have 
\begin{align*}
\delta(A \cup B) & \leq \delta(A \cup \{x\}) + \delta(\{x,y\}) + \delta(B \cup \{y\}) \\
& = \delta(A) + d(x,y) + \delta(B)
\end{align*}
so that $D_\delta(A,B) \leq \ell$.

To show that also $\ell \leq D_\delta(A,B)$, first consider $r_A = \delta(A)$ and $r_B = D_\delta(A,B) + \delta(B)$. Then we have $\delta(A) \leq r_A$, $\delta(B) \leq r_B$ and $\delta(A \cup B) \leq r_A + r_B$. Thus, since $(X,\delta)$ is hyperconvex, there exists $x^* \in X$ such that $\delta(A \cup \{x^*\}) \leq r_A = \delta(A)$ and $\delta(B \cup \{x^*\}) \leq r_B =  D_\delta(A,B) + \delta(B)$. 
In particular, $x^* \in \sB_A$.

Next consider $r_{\{x^*\}} = D_{\delta}(A,B)$ and $r_B = \delta(B)$. Then we have $\delta(\{x^*\}) \leq r_{\{x^*\}}$, $\delta(B) \leq r_B$ and $\delta(B \cup \{x^*\}) \leq D_\delta(A,B) + \delta(B) = r_{\{x^*\}} + r_B$. Thus, again in view of the fact that $(X,\delta)$ is hyperconvex, there exists $y^* \in X$ such that $d(x^*,y^*) = \delta(\{x^*,y^*\}) \leq r(\{x^*\}) = D_\delta(A,B)$ and $\delta(B \cup \{y^*\}) \leq r_B = \delta(B)$. Hence, $y^* \in \sB_B$ and $d(x^*,y^*) \leq D_\delta(A,B)$, as required.
\end{proof}

We can now prove our characterization theorem for arboreal diversities. As above, for each $x \in X$, we let $g_x$ denote the map given by $g_x(A) = \delta(A \cup \{x\})$ for all $A \in \Pf(X)$. To apply Corollary~\ref{cor:ba:path:connected} and Proposition~\ref{prop:embedD}, we define, for any diversity $(X,\delta)$ and any $A \in \Pfs(X)$, $G(A) := \{g_x : x \in A\}$. Then we have, in view of~\eqref{eq:embed:div:t:delta}, $\delta_{T}(G(A)) = \delta(A)$. Hence, by \cite[Thm.~2.8]{BryantTupper12}, we have
\begin{equation}
\label{eq:bga:vs:kappaa}
\sB_{G(A)} = \{f \in T_{\delta} : \delta_{T}(G(A) \cup \{f\}) = \delta(A)\} = \{f \in T_{\delta} : f(A) = \delta(A)\}
\end{equation}
for all $A \in \Pfs(X)$.
In addition, we put $\sB_{G(\emptyset)} = T_{\delta}$.

\begin{thm}
\label{thm:char:arboreal}
Let $(X,\delta)$ be a diversity and $D = D_{\delta}$. Then $(X,\delta)$ is arboreal if and only if the associated distance space $(\Pf(X),D)$ satisfies the extended four-point condition~\eqref{eq:extended4pt}.
\end{thm}

\begin{proof}
In view of Lemma~\ref{lem:arboreal:sufficient} it remains to show that if $(X,\delta)$ is arboreal then $(\Pf(X),D)$ satisfies~\eqref{eq:extended4pt}. So, consider the diversity $(T_{\delta},\delta_{T})$ and assume that the metric space $(T_\delta,\dinf)$ is a real tree. By \cite[Thm.~3.6]{BryantTupper12}, $(T_{\delta},\delta_{T})$ is a hyperconvex diversity. Therefore, by Lemma~\ref{lem:div:hyperconvex:path:connected}, $(T_\delta,\dinf)$ is complete and, by Proposition~\ref{prop:embedD}, we have $D(A,B) = \inf \{\dinf(f,h) : f \in \sB_{G(A)}, \ h \in \sB_{G(B)}\}$ for all $A,B \in \Pf(X)$. Moreover, in view of Corollary~\ref{cor:ba:path:connected}, $\sB_{G(A)}$ is a geodesic and complete subset of the real tree $(T_\delta,\dinf)$. Thus, $(T_\delta,\dinf)$ together with the family $\{\sB_{G(A)}\}_{A \in \Pf(X)}$ is a subtree representation of $(\Pf(X),D)$. This implies, by Theorem~\ref{thm:char:like:hirai}, that $(\Pf(X),D)$ satisfies~\eqref{eq:extended4pt}, as required.
\end{proof}

We conclude this section with a characterization of phylogenetic diversities among arboreal diversities that follows immediately from \cite[Thm.~5.7]{BryantTupper12}.

\begin{cor}
\label{cor:char:phylogenetic}
Let $(X,\delta)$ be an arboreal diversity. Then $(X,\delta)$ is a phylogenetic diversity if and only if, for all $A \in \Pf(X) \setminus \{\emptyset,X\}$,
\[\sB_{G(A)} = \{f \in T_{\delta} : f \ \text{is on a geodesic between} \ g_x \ \text{and} \ g_y \ \text{for some} \ x,y \in A\}.\]
\end{cor}

\subsection*{Acknowledgments}

This research was supported in part by an International Exchanges award from The Royal Society (UK) to KTH, VM and DB. The authors thank the referee for their many insightful and helpful comments.

\bibliographystyle{acm}
\bibliography{ArborealBib}

\begin{thebibliography}{10}

\bibitem{Aronszajn56}
{\sc Aronszajn, N., and Panitchpakdi, P.}
\newblock Extension of uniformly continuous transformations and hyperconvex
  metric spaces.
\newblock {\em Pacific Journal of Mathematics 6\/} (1956), 405--439.

\bibitem{BandeltDress92}
{\sc Bandelt, H.-J., and Dress, A. W.~M.}
\newblock A canonical decomposition theory for metrics on a finite set.
\newblock {\em Advances in Mathematics 92}, 1 (1992), 47--105.

\bibitem{BryantMoulton04}
{\sc Bryant, D., and Moulton, V.}
\newblock {NeighborNet}: An agglomerative algorithm for the construction of
  planar phylogenetic networks.
\newblock {\em Molecular Biololgy and Evolution 21\/} (2004), 255--265.

\bibitem{BryantTupper12}
{\sc Bryant, D., and Tupper, P.~F.}
\newblock Hyperconvexity and tight-span theory for diversities.
\newblock {\em Advances in Mathematics 231}, 6 (Dec. 2012), 3172--3198.

\bibitem{BryantTupper14}
{\sc Bryant, D., and Tupper, P.~F.}
\newblock Diversities and the geometry of hypergraphs.
\newblock {\em Discrete Math. Theor. Comput. Sci. 16}, 2 (2014), 1--20.

\bibitem{DescombesPavon17}
{\sc Descombes, D., and Pav\'on, M.}
\newblock Injective subsets of $\ell_\infty(i)$.
\newblock {\em Advances in Mathematics 317}, 7 (2017), 91--107.

\bibitem{develin2004tropical}
{\sc Develin, M., and Sturmfels, B.}
\newblock Tropical convexity.
\newblock {\em Documenta Mathematica 9\/} (2004), 1--27.

\bibitem{dress2011basic}
{\sc Dress, A., Huber, K.~T., Koolen, J., Moulton, V., and Spillner, A.}
\newblock {\em Basic phylogenetic combinatorics}.
\newblock Cambridge University Press, 2011.

\bibitem{dress1996t}
{\sc Dress, A., Moulton, V., and Terhalle, W.}
\newblock T-theory: an overview.
\newblock {\em European Journal of Combinatorics 17}, 2-3 (1996), 161--175.

\bibitem{Dress84}
{\sc Dress, A. W.~M.}
\newblock Trees, tight extensions of metric spaces, and the cohomological
  dimension of certain groups: a note on combinatorial properties of metric
  spaces.
\newblock {\em Advances in Mathematics 53}, 3 (1984), 321--402.

\bibitem{Edgar08}
{\sc Edgar, G.}
\newblock {\em Measure, topology, and fractal geometry}, second~ed.
\newblock Undergraduate Texts in Mathematics. Springer, New York, 2008.

\bibitem{EspinolaKhamsi01}
{\sc Esp{\'{\i}}nola, R., and Khamsi, M.~A.}
\newblock Introduction to hyperconvex spaces.
\newblock In {\em Handbook of metric fixed point theory}. Kluwer Acad. Publ.,
  Dordrecht, 2001, pp.~391--435.

\bibitem{EspinolaKirk06}
{\sc Esp{\'{\i}}nola, R., and Kirk, W.~A.}
\newblock Fixed point theorems in {$\Bbb R$}-trees with applications to graph
  theory.
\newblock {\em Topology and its Applications 153}, 7 (2006), 1046--1055.

\bibitem{EspinolaPiatek14}
{\sc Esp{\'{\i}}nola, R., and Pi{\c{a}}tek, B.}
\newblock Diversities, hyperconvexity and fixed points.
\newblock {\em Nonlinear Analysis 95\/} (2014), 229--245.

\bibitem{Faith92}
{\sc Faith, D.}
\newblock Conservation evaluation and phylogenetic diversity.
\newblock {\em Biological Conservation 61\/} (1992), 1--10.

\bibitem{HerrmannMoulton12}
{\sc Herrmann, S., and Moulton, V.}
\newblock Trees, tight-spans and point configurations.
\newblock {\em Discrete Mathematics 312}, 16 (2012), 2506--2521.

\bibitem{Hirai06}
{\sc Hirai, H.}
\newblock Characterization of the distance between subtrees of a tree by the
  associated tight span.
\newblock {\em Annals of Combinatorics 10\/} (2006), 111--128.

\bibitem{hirai2006geometric}
{\sc Hirai, H.}
\newblock A geometric study of the split decomposition.
\newblock {\em Discrete \& Computational Geometry 36\/} (2006), 331--361.

\bibitem{Hirai09}
{\sc Hirai, H.}
\newblock Tight spans of distances and the dual fractionality of undirected
  multiflow problems.
\newblock {\em Journal of Combinatorial Theory, Series B 99}, 6 (2009),
  843--868.

\bibitem{hirai2011folder}
{\sc Hirai, H.}
\newblock Folder complexes and multiflow combinatorial dualities.
\newblock {\em SIAM Journal on Discrete Mathematics 25}, 3 (2011), 1119--1143.

\bibitem{HusonBryant06}
{\sc Huson, D., and Bryant, D.}
\newblock Application of phylogenetic networks in evolutionary studies.
\newblock {\em Molecular Biology and Evolution 23\/} (2006), 254--267.

\bibitem{Isbell64}
{\sc Isbell, J.~R.}
\newblock Six theorems about injective metric spaces.
\newblock {\em Commentarii Mathematici Helvetici 39\/} (1964), 65--76.

\bibitem{janson2023real}
{\sc Janson, S.}
\newblock Real trees.
\newblock {\em arXiv preprint arXiv:2303.07920\/} (2023).

\bibitem{KirkShahzad14}
{\sc Kirk, W., and Shahzad, N.}
\newblock Diversities.
\newblock In {\em Fixed Point Theory in Distance Spaces}, W.~Kirk and
  N.~Shahzad, Eds. Springer International Publishing, Cham, 2014, pp.~153--158.

\bibitem{KS19a}
{\sc Kirk, W., and Shahzad, N.}
\newblock Hyperbolic spaces and directional contractions.
\newblock {\em Bulletin of Mathematical Sciences 9}, 03 (2019), 1950021.

\bibitem{lang2013injective}
{\sc Lang, U.}
\newblock Injective hulls of certain discrete metric spaces and groups.
\newblock {\em Journal of Topology and Analysis 5}, 03 (2013), 297--331.

\bibitem{maehara2022optimal}
{\sc Maehara, T., and Ando, K.}
\newblock Optimal algorithm for finding representation of subtree distance.
\newblock {\em IEICE Transactions on Fundamentals of Electronics,
  Communications and Computer Sciences 105}, 9 (2022), 1203--1210.

\bibitem{Pereira69}
{\sc Pereira, J.~S.}
\newblock A note on the tree realizability of a distance matrix.
\newblock {\em Journal of Combinatorial Theory 6\/} (1969), 303--310.

\bibitem{SmythTsaur02}
{\sc Smyth, M.~B., and Tsaur, R.}
\newblock Hyperconvex semi-metric spaces.
\newblock {\em Topology Proceedings 26\/} (2002), 791--801.

\bibitem{cech1966}
{\sc \u{C}ech, E., Frol\'{i}k, Z., and Kat\u{e}tov, M.}
\newblock {\em Topological spaces}.
\newblock Academia, Publishing House of the Czechoslovak Academy of Sciences,
  1966.

\end{thebibliography}

\end{document}